\documentclass[10pt,a4paper]{article}
\usepackage[latin1]{inputenc}
\usepackage{amsmath}
\usepackage{amsfonts}
\usepackage{amssymb}
\usepackage{mathrsfs}
\usepackage{amsthm}
\usepackage{makeidx}

\numberwithin{equation}{section}
\newtheorem{os}{Remark}[section]
\newtheorem{te}{Theorem}[section]

\allowdisplaybreaks

\usepackage[pdftex]{graphicx}
\usepackage[margin=3cm]{geometry}

\title{Vibrations and fractional vibrations of rods, plates and Fresnel pseudo-processes}

\author{ORSINGHER Enzo\footnote{Corresponding author} \hspace{1cm} D'OVIDIO Mirko\\ {\small Dipartimento di Scienze Statistiche, ''Sapienza'' University of Rome}\\ {\small enzo.orsingher@uniroma1.it  \hspace{.5cm}  mirko.dovidio@uniroma1.it}}

\begin{document}
\maketitle

\begin{center}
\begin{minipage}{5in}
\textbf{Abstract} Different initial and boundary value problems for the equation of vibrations of rods (also called Fresnel equation) are solved by exploiting the connection with Brownian motion and the heat equation. The analysis of the fractional version (of order $\nu$) of the Fresnel equation is also performed and, in detail, some specific cases, like $\nu=1/2$, $1/3$, $2/3$, are analyzed. By means of the fundamental solution of the Fresnel equation, a pseudo-process $F(t)$, $t>0$ with real sign-varying density is constructed and some of its properties examined. The equation of vibrations of plates is considered and the case of circular vibrating disks $C_R$ is investigated by applying the methods of planar orthogonally reflecting Brownian motion within $C_R$. The composition of F with reflecting Brownian motion $B$ yields the law of biquadratic heat equation while the composition of $F$ with the first passage time $T_t$ of $B$ produces a genuine probability law strictly connected with the Cauchy process.
\end{minipage}
\end{center}


\tableofcontents

\section{Introduction}
One of the most important equations of mathematical physics is that of the vibrations of rods (or equivalently of beams and shafts) which can be written as
\begin{equation} 
\frac{\partial^2 u}{\partial t^2} = - \kappa^2 \frac{\partial^4 u}{\partial x^4}, \quad x \in \mathbb{R},\, t>0. \label{eq0}
\end{equation}
The constant $\kappa$ appearing in \eqref{eq0} is related to the physical structure of the vibrating rod and will be considered below as equal to $\kappa=1/2$ for reasons which will appear clearly further in the text. For the derivation of \eqref{eq0} consult \cite[pages 112 - 114]{due} or the classical book by Courant and Hilbert \cite[pages 244 - 246]{CouHilb} where, also the derivation of the equation of the vibrations of plates (pages 250 - 252)
\begin{equation}
\frac{\partial^2 u}{\partial t^2} = - \rho \left[ \frac{\partial^2}{\partial x^2} + \frac{\partial^2}{\partial y^2} \right]^2 u 
\end{equation}
is presented ($\rho$ is a physical constant assumed below equal to $\rho = 1/2^2$). The equation \eqref{eq0} can be written as
\begin{equation}
\left( \frac{\partial}{\partial t} + \frac{i}{2} \frac{\partial^2}{\partial x^2} \right) \left( \frac{\partial}{\partial t} - \frac{i}{2} \frac{\partial^2}{\partial x^2} \right) u = 0
\end{equation}
and this shows a strict connection with the Schr\"odinger equation and therefore with the heat equation. In effect
\begin{equation}
\frac{\partial u}{\partial t} = \pm \frac{i}{2} \frac{\partial^2 u}{\partial x^2}
\end{equation}
can be easily reduced by means of a time change $t^\prime = \pm i t$ to the classical homogeneous heat equation. This connection entails that initial-boundary value problems concerning equation \eqref{eq0} have solutions which can be constructed by means of related solutions of problems concerning the heat equation and therefore are related to the standard Brownian motion $B$. In particular, the fundamental solution of \eqref{eq0} can be obtained by means of the rule
\begin{align}
u(x, t)dx = & \frac{1}{2} \left[ Pr \{ B(s) \in dx \} \Big|_{s=it} + Pr\{ B(s) \in dx \} \Big|_{s = -it} \right] =  \frac{1}{2} \left[ \mu\{dx, t \} + \mu\{dx, -it \} \right]\nonumber \\
= &  \frac{dx}{2} \left[ \frac{e^{- \frac{x^2}{2it} - i \frac{\pi}{4}}}{\sqrt{2\pi t}} + \frac{e^{-\frac{x^2}{-2it} + i \frac{\pi}{4}} }{\sqrt{2\pi t}} \right] =  \frac{dx}{\sqrt{2\pi t}} \cos\left( \frac{x^2}{2t} - \frac{\pi}{4} \right). \label{funQW}
\end{align}
The measure $\mu$ appearing in \eqref{funQW} must be understood in the sense that 
\[ \mu\{ dx, \pm it \} = Pr\{ B(s) \in dx\} \Big|_{s= \pm i t}. \] 
The idea underlying \eqref{funQW} is that at time $t=0$ a Brownian motion is started off from $x=0$ either with increasing or decreasing imaginary time. The choice of time direction is made once only and cannot be changed. This process can be regarded as the limit of a random walk with symmetric real-valued steps separated by imaginary time intervals. Each term in \eqref{funQW} is therefore connected with a sort of Brownian motion where time takes imaginary values. This approach permits us to exploit the panoplie of mathematical results and instruments of Brownian motion to obtain solutions of all possible types of problems for vibrations of infinite, semi infinite and finite rods. Furthermore, the finite, signed density of the measure \eqref{funQW} suggests the construction of pseudo-processes (called Fresnel pseudo-process and denoted by $F(t)$, $t>0$) whose finite-dimensional distributions can be constructed according to the rule
\begin{align}
u(x_1,t_1 \ldots , x_n, t_n) \prod_{j=1}^{n}dx_j = & \frac{1}{2}\bigg[Pr\{ B(s_1) \in dx_1, \ldots , B(s_n) \in dx_n \} \Big|_{s_j=it_j}\nonumber \\ 
+ & Pr\{ B(s_1) \in dx_1, \ldots , B(s_n) \in dx_n \}\Big|_{s_j=-it_j} \bigg] \label{funQW2}
\end{align}
$j=1,2,\ldots , n$. The construction of Wiener-type measures for higher-order heat-type equations
\begin{equation}
\frac{\partial u}{\partial t} = \pm \frac{\partial^n u}{\partial x^n} \label{eqho}
\end{equation}
has been realized by different authors since the beginning of the Sixties (see Krylov \cite{Kry60}, Ladokhin \cite{Lad63}). Much work on pseudo-processes connected with \eqref{eqho} has recently been done by Lachal \cite{Lach03, Lach07}, Cammarota and Lachal \cite{LachCam}, where by different means and techniques, various types of functionals of these pseudo-processes have been investigated.

The fractional version of the equation of vibrations of rods
\begin{equation}
\frac{\partial^{2\nu} u}{\partial t^{2\nu}} = - \frac{1}{2^2} \frac{\partial^4 u}{\partial x^4}, \quad 0 < \nu \leq 1 \label{fracpdeINT}
\end{equation} 
is also examined. The Fourier transform of the solution of \eqref{fracpdeINT}, equipped with the necessary initial conditions, takes the form
\begin{align}
U_{2\nu}(\beta, t) = & \int_{-\infty}^{+\infty} e^{i\beta x} u_{2\nu}(x, t)dx =  \frac{1}{2} \left[ E_{\nu, 1}\left( i \frac{\beta^2 t^\nu}{2} \right) + E_{\nu, 1}\left( -i \frac{\beta^2 t^\nu}{2} \right) \right] =  E_{2\nu, 1}\left( \frac{\beta^4 t^{2\nu}}{2^2} \right) \nonumber
\end{align}
where 
\[ E_{\nu,1}(x) = \sum_{k=0}^{\infty} \frac{x^k}{\Gamma(\nu k +1)} \]
is the Mittag-Leffler function. We are able to invert the Fourier transform above and obtain an explicit solution of the fractional equation \eqref{fracpdeINT} in the form
\begin{equation}
u_{2\nu}(x,t) = \frac{1}{\pi \sqrt{2t^{\nu}}} \sum_{m=0}^\infty \frac{1}{m!}\left(- \sqrt{2}\frac{|x|}{\sqrt{t^\nu}} \right)^{m}\cos \left( \frac{m+1}{4}\pi \right) \sin\left(\frac{m+1}{2}\pi \nu \right)\Gamma\left( \frac{m+1}{2}\nu \right). \label{sumsolINT}
\end{equation}
From \eqref{sumsolINT} for $\nu=1$ we obtain the fundamental solution \eqref{funQW} while for $\nu=1/2$ we get one alternative expression for the solution of the biquadratic heat equation. This permits us to give a sketch of the function $u_1(x,t)$, $x \in \mathbb{R}$, $t>0$ which looks like the normal bell-shaped density in the neighborhood of the origin (see figure \ref{lab4ord}). In great detail we examine the function $u_{4/3}$ ($\nu=2/3$) and prove that it is possible to reduce it to the superposition of Airy functions. Also, the case $\nu=1/3$ is analyzed and we prove that it can be expressed in terms of \eqref{funQW} and the Airy function.

The analysis of Section 3 shows that the profile of the vibrating rod has a peak near the origin where the initial disturbance is originated and symmetric damping waves which rapidely decrease in size. In our opinion this accords with what happens with real vibrations where friction dissipates the original energy and therefore fractional equations of vibrations of rods better describe the phenomenon under investigation, see figure \ref{figall}. 

Section 5 is devoted to the multidimensional version of the equation \eqref{eq0} which governs the vibrations of plates. Its general form is
\begin{equation}
\frac{\partial^2 u}{\partial t^2} = - \frac{1}{2^2} \left( \sum_{j=1}^d \frac{\partial^2}{\partial x_j^2} \right) u
\end{equation}
and possesses the fundamental solution of the form
\begin{equation}
u(x_1, \ldots , x_d, t) = \frac{1}{\left( \sqrt{2\pi t}\right)^d} \cos \left( \sum_{j=1}^d \frac{x_j^2}{2t} - d \frac{\pi}{4} \right).
\end{equation}
In the plane we study the vibrations of a circular plate $C_R$ with the Neumann boundary condition on the edge $\partial C_R$. We solve this problem by applying an approach based on the inversion with respect to the circle which parallels the construction of the reflecting planar Brownian motion inside a disk. This is useful to describe the oscillations of thin vibrating structures, started off by a concentrated central initial disturbance. 

Finally, we show that the composition of the Fresnel pseudo-process $F(t)$, $t>0$ with reflecting Brownian motion $|B(t)|$ produces a subordinated process $F(|B(t)|)$ whose one-dimensional law coincides with the fundamental solution to
\begin{equation}
\left\lbrace \begin{array}{l} \frac{\partial u}{\partial t} = - \frac{1}{2^3} \frac{\partial^4 u}{\partial x^4},\\ u(x,0) = \delta(x). \end{array} \right .
\end{equation}
Instead, the composition of $F$ with $T_t = \inf\{ s \geq 0\,:\, B(s)=t \}$ yields the following interesting genuine probability law
\begin{equation}
Pr\{ F(T_t) \in dx \}/dx = \frac{t}{\pi \sqrt{2}} \frac{t^2 + x^2}{t^4 + x^4}, \quad x \in \mathbb{R},\, t>0 \label{disGenuine}
\end{equation}
which solves the fourth-order equation
\begin{equation}
\frac{\partial^4 u}{\partial t^4} + \frac{\partial^4 u}{\partial x^4} = 0
\end{equation}
The distribution \eqref{disGenuine} has two maxima which move in opposite directions as time passes and display a structure similar to the solutions of fractional diffusion equations for a degree of fractionality varying in $1 < \nu < 2$ (see \cite{Fuj90}, \cite{OB09}). Finally the successive compositions of Fresnel pseudo-processes and the law of
\begin{equation}
\mathcal{F}_n(t) = F_1(|F_2(|\ldots F_{n+1}(t) \ldots |)|), \quad t>0
\end{equation}
is shown to be governed by the higher-order equation
\begin{equation}
\frac{\partial^2 u}{\partial t^2}(x,t) = - 2^{-2(2^{n+1} -1)} \frac{\partial^{2^{n+2}}u}{\partial x^{2^{n+2}}}(x,t).
\end{equation}
This is clearly related to what happens with the n-th order iterated Brownian motion as shown in Orsingher and Beghin \cite{OB09}, or also by Baeumer et al. \cite{BMN09} for fractional equations. For higher-order equations, see also D'Ovidio and Orsingher \cite{DO2}.

\subsection{Notations}

\begin{itemize}
\item $u(x,t) = \frac{1}{\sqrt{2\pi t}} \cos\left( \frac{x^2}{2t} - \frac{\pi}{4} \right)$ : fundamental solution of the one-dimensional Fresnel equation,
\item $U(\beta, t) = \int_{-\infty}^{+\infty} e^{i\beta x} u(x,t) dx$,
\item $u^a$ is the solution to the Fresnel equation with absorbing conditions,
\item $u^r$ is the solution to the Fresnel equation with reflecting conditions,
\item $u^{el}$ is the solution to the Fresnel equation with elastic boundary conditions,
\item $p^{el}$ is density of the elastic Brownian motion,
\item $q$ is the fundamental solution to the biquadratic equation.
\item $p^{ref}(r,t)$ is the probability density of the reflecting Brownian motion in a disk $C_R$ ($q^{ref}$ is its kernel),
\item $\bar{p}^{ref}$ is the law of a vibrating plate of radius $R$ with Neumann condition on the border $\partial C_R$ ($\bar{q}^{ref}$ is its kernel),
\item $u_{2\nu}$ is the fundamental solution of the fractional Fresnel equation \eqref{fracpdeINT},
\item $v_{\nu}$ is the fundamental solution to the fractional diffusion equation \eqref{fracdifeqP}.
\end{itemize}
The general unknown functions of the equations appearing in the paper (also the fractional equations) are indicated by $u$ while the solutions of the boundary Cauchy problems are denoted by $u^a$, $u^{el}$, etc.. The solutions of the initial-value problems for fractional Fresnel equations are denoted by $u_{2\nu}$ and for the fractional diffusion equation by $v_\nu$. 
\begin{itemize}
\item $u(x_1, \ldots , x_d, t) = \frac{1}{(\sqrt{2\pi t})^d} \cos\left( \sum_{j=1}^d \frac{x^2_j}{2t} - d\frac{\pi}{4} \right)$ : fundamental solution of the $n$-dimensional Fresnel equation,
\item $U(\beta_1, \ldots , \beta_d, t) = \int_{-\infty}^{+\infty} dx_1 \ldots \int_{-\infty}^{+\infty} dx_d e^{i \beta \sum_{j=1}^d \beta_j x_j } u(x_1, \ldots , x_d, t) $,
\item \[ u(x_1,t_1; \ldots x_n,t_n) = \frac{(2\pi)^{-n/2}}{\prod_{j=1}^n \sqrt{t_j - t_{j-1}} }\cos\left( \sum_{j=1}^n \frac{(x_j - x_{j-1})^2}{2 (t_j-t_{j-1})} - n\frac{\pi}{4} \right)  \]
is the $n$-dimensional measure density of the Fresnel pseudo-process $F(t)$, $t>0$,
\end{itemize}

\begin{itemize}
\item $B(t)$, $t>0$ is a standard Brownian motion,
\item $F(t)$, $t>0$ is the Fresnel pseudo-process, 
\item $T_t=\inf\{ s \geq 0\,:\, B(s) =t \}$.
\end{itemize}

\section{Vibrations of rods and Brownian motion}
The equation of vibrations of a thin, flexible rod has the following form
\begin{equation} 
\frac{\partial^2 u}{\partial t^2} = - \frac{1}{2^2} \frac{\partial^4 u}{\partial x^4}, \quad x \in \mathbb{R},\, t>0. \label{eq1}
\end{equation}
Vibrations of rods differ substantially from vibrations of strings (and membranes in dimension 2) because strings propagate with infinite velocity as heat does. Furthermore, the heat diffusion and the vibrations of rods have strict connections which inspire much of the work of this section.

Since \eqref{eq1} can be written  as
\begin{equation} 
\left( \frac{\partial}{\partial t} + \frac{i}{2} \frac{\partial^2}{\partial x^2} \right) \left( \frac{\partial}{\partial t} - \frac{i}{2} \frac{\partial^2}{\partial x^2} \right)u=0 \label{similarPDE}
\end{equation}
we can decouple the original equation into two Schr\"odinger equations (each of which can be reduced to heat equations by a suitable change of variable, $t^\prime = \pm i t$). The fundamental solution to \eqref{eq1} can be obtained by suitably combining the solutions to
\begin{equation}
\left( \frac{\partial}{\partial t} + \frac{i}{2} \frac{\partial^2}{\partial x^2} \right) v_1= 0, \label{sup1}
\end{equation}
\begin{equation}
\left( \frac{\partial}{\partial t} - \frac{i}{2} \frac{\partial^2}{\partial x^2} \right) v_2= 0. \label{sup2}
\end{equation}
The fundamental solutions to \eqref{sup1} and \eqref{sup2} can be written as
\begin{equation}
\left\lbrace \begin{array}{l} v_1(x,t) = \frac{e^{-i \frac{\pi}{4}}}{\sqrt{2 \pi t}} \left[ \cos \frac{x^2}{2t} + i \sin \frac{x^2}{2t} \right] \\ v_2(x,t) = \frac{e^{i \frac{\pi}{4}}}{\sqrt{2 \pi t}} \left[ \cos \frac{x^2}{2t} - i \sin \frac{x^2}{2t} \right]  \end{array} \right .
\end{equation}
and thus
\begin{align}
u(x,t) = & \frac{1}{2}\Big[ v_1(x,t) + v_2(x,t) \Big] =  \frac{1}{\sqrt{4 \pi t}} \Big[ \cos \frac{x^2}{2t} + \sin \frac{x^2}{2t} \Big] =  \frac{1}{\sqrt{2\pi t}} \cos\left( \frac{x^2}{2t} - \frac{\pi}{4} \right). \label{fun25}
\end{align}
We call \eqref{fun25} Fresnel function and can easily check that
\[ \int_{\mathbb{R}} u(x,t) dx = 1 \]
because 
\begin{equation} 
\int_0^\infty \cos x^2 dx=\int_0^\infty \sin x^2 dx = \frac{1}{2} \sqrt{\frac{\pi}{2}}.
\end{equation}
The function \eqref{fun25} is the solution to the Cauchy problem
\begin{equation}
\left\lbrace \begin{array}{ll} \frac{\partial^2 u}{\partial t^2} = - \frac{1}{2^2} \frac{\partial^4 u}{\partial x^4}\\ u(x,0)=\delta(x)\\ u_t(x,0)=0 \end{array} \right .
\label{pdeProbtm}
\end{equation}
and its Fourier transform reads
\begin{equation}
U(\beta, t) = \int_{-\infty}^{+\infty} e^{i \beta x} u(x, t) dx=\cos \frac{\beta^2 t}{2}. \label{jkl}
\end{equation} 
This can be proved by writing
\begin{align*}
\frac{1}{\sqrt{2\pi t}} \int_{-\infty}^{+\infty} e^{i \beta x} \cos\left( \frac{x^2}{2t} - \frac{\pi}{4} \right)dx = & \frac{1}{2} e^{-i \frac{\pi}{4} - i \frac{\beta^2 t}{2}} \int_{-\infty}^{+\infty} \frac{e^{i \frac{(x + \beta)^2}{2t}}}{\sqrt{2\pi t}} dx \\
+ & \frac{1}{2} e^{i \frac{\pi}{4} + i \frac{\beta^2 t}{2}} \int_{-\infty}^{+\infty} \frac{e^{-i \frac{(x + \beta)^2}{2t}}}{\sqrt{2\pi t}} dx\\
= & \frac{2\sqrt{2}}{\sqrt{\pi}} \cos \frac{\beta^2 t}{2} \int_{0}^{\infty} \cos y^2\, dy = \cos \frac{\beta^2 t}{2}.
\end{align*}
Result \eqref{jkl} can be also obtained by solving the Fourier transform of problem \eqref{pdeProbtm}, that is
\begin{equation}
\left\lbrace \begin{array}{ll} \frac{\partial^2 U}{\partial t^2} = - \frac{\beta^4}{2^2} U(\beta, t),\\ U(\beta,0)=1,\\ U_t(\beta,0)=0. \end{array} \right .
\end{equation}
The structure of the function \eqref{fun25} is illustrated in figure \ref{figgg}, and it represents the profile of the vibrating rod at different times. It shows that the deformation instantaneously propagates along the $x$ axis and the oscillations decay as time passes.
\begin{figure}[ht]
\centering
\includegraphics[scale=.5]{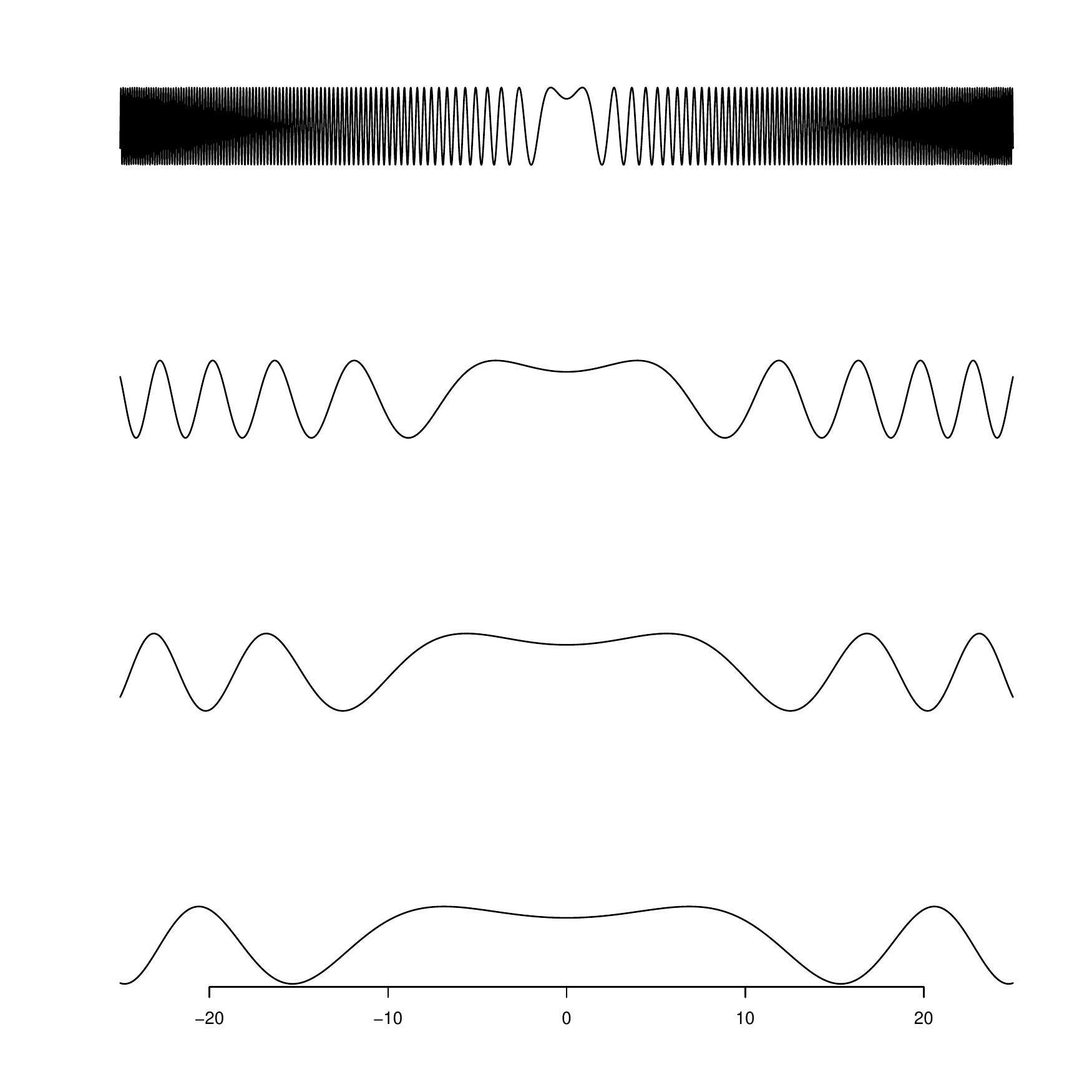} 
\caption{The Fresnel function \eqref{funQW} represents the propagation of vibrations along an infinite rod at time $t=1,20,40,60$.}
\label{figgg}
\end{figure}
The measure density \eqref{fun25} is a sign-varying function and plays an important role in the theory of optical diffraction. By studying diffusion processes with an alternating time direction (forward or backwards) an equation similar to \eqref{similarPDE} was obtained and its solutions investigated in Orsingher \cite{ORS86}.

\begin{os}
\normalfont
The area between two successive roots $\alpha_k$, $\alpha_{k+1}$ of the Fresnel function 	\eqref{funQW} tends to zero as $k \to \infty$. This can be shown by observing that
\begin{align*}
\Bigg| \frac{1}{\sqrt{2\pi t}} \int_{\alpha_k}^{\alpha_{k+1}}  \cos\left( \frac{x^2}{2t} - \frac{\pi}{4} \right)  dx \Bigg| \leq & \Bigg| \frac{1}{\alpha_k} \sqrt{\frac{t}{2\pi}} \int_{\alpha_k}^{\alpha_{k+1}} \frac{x}{t} \cos\left( \frac{x^2}{2t} - \frac{\pi}{4} \right) dx \Bigg|  \\
 \leq & \frac{1}{\alpha_k} \sqrt{\frac{t}{2\pi}} \Bigg| \sin\left( \frac{\alpha_k^2}{2t} -\frac{\pi}{4} \right) - \sin\left( \frac{\alpha_{k+1}^2}{2t} -\frac{\pi}{4} \right) \Bigg|
\end{align*}
where $\alpha_k = \sqrt{2\pi} \sqrt{3/2 + k}$.
\end{os}

By exploiting the ideas outlined above it is possible to study the vibrations of semiinifinite rods with different types of boundary conditions. We start by taking into account the vibrations with an absorbing condition at the free end point $x=0$. The rod is assumed to be clamped at infinity. We arrive now at our first Theorem.
\begin{te}
The solution to the boundary-value problem
\begin{equation}
\left\lbrace \begin{array}{l} \frac{\partial^2 u}{\partial t^2}(x,t) = - \frac{1}{2^2} \frac{\partial^4 u}{\partial x^2}(x,t) , \quad x>0,\, t>0 \\
u(x,0) = \delta(x-y)\\
u(0,t)=0\\
u_t(0,t) = 0 \end{array} \right .
\label{PDEabsorbing}
\end{equation}
is given by
\begin{equation}
u^a(x,t;y,0) = \frac{1}{\sqrt{2\pi t}} \left[ \cos\left( \frac{(x-y)^2}{2t} - \frac{\pi}{4} \right) - \cos\left( \frac{(x+y)^2}{2t} - \frac{\pi}{4} \right) \right]. \label{urod2}
\end{equation}
\end{te}
\begin{proof}
Our scheme consists in considering two absorbing Brownian motions, one with imaginary increasing time and one with imaginary decreasing time. Thus, we can write
\begin{equation}
u^a(x,t;y,0) = \frac{1}{2} \left[ p(x, it; y,0) + p(x, -it; y,0) \right] \label{tmlaa}
\end{equation}
where $p(x,s;y,0)$ is the transition function of a standard Brownian motion. Each component of \eqref{tmlaa} solves the boundary-value problem for the Schr\"odinger equation
\[ \frac{\partial u}{\partial t} = \pm i\frac{\partial^2 u}{\partial x^2}.\]
From \eqref{tmlaa} some calculations lead to \eqref{urod2}. 
\end{proof}
We can observe that \eqref{urod2} solves the equation of vibrations of rods and the boundary conditions can be checked by setting $x=0$. In the same spirit we can solve much more complicated boundary-value problems such as those involving the elastic condition at the free end point of the vibrating rod. Elastic Brownian motion was studied by It\^o-McKean  \cite{IK63}. The fractional extension of elastic Brownian motion was investigated by Beghin and Orsingher \cite{BO09spa}. Our result is in the next Theorem.
\begin{te}
The solution to the elastic boundary-value problem
\begin{equation}
\left\lbrace \begin{array}{l} \frac{\partial^2 u}{\partial t^2}(x,t) = -\frac{1}{2^2} \frac{\partial^4 u}{\partial x^4}(x,t), \quad x >0, \; t>0\\
u(x,0) = \delta(x-y)\\
u_t(x,0)=0\\
u_x - \alpha u |_{x=0} =0,\quad \alpha >0 \end{array} \right .
\label{PDEelastic}
\end{equation}
is given by 
\begin{align}
u^{el}(x,t;y,0) = & \frac{1}{\sqrt{2\pi t}} \left[ \cos\left( \frac{(x-y)^2}{2t} - \frac{\pi}{4} \right) - \cos\left(\frac{(x+y)^2}{2t} - \frac{\pi}{4} \right) \right] \nonumber \\
 & + 2 e^{\alpha(x+y)} \int_{x+y}^{+\infty} \frac{w\, e^{-\alpha w}}{\sqrt{2\pi t^3}} \cos\left( \frac{w^2}{2t}  - \frac{3 \pi}{4}\right)\, dw \label{eqelastic1}\\
= & \frac{1}{\sqrt{2\pi t}} \left[ \cos\left( \frac{(x-y)^2}{2t} - \frac{\pi}{4} \right) + \cos\left( \frac{(x+y)^2}{2t} - \frac{\pi}{4} \right)\right]\nonumber \\
& - \frac{2 \alpha e^{\alpha (x+y)}}{\sqrt{2\pi t}} \int_{x+y}^{\infty} e^{-\alpha w} \cos\left( \frac{w^2}{2t} - \frac{\pi}{4} \right) dw. \label{eqelastic2}
\end{align}
\end{te}
\begin{proof}
The solution to \eqref{PDEelastic} can be written as
\begin{equation}
u^{el}(x,t;y,0) = \frac{1}{2}\left[ p^{el}(x, it; y,0)  + p^{el}(x, -it; y,0) \right]
\end{equation}
where $p^{el}(x, s; y,0)$ is the solution to the heat equation with the elastic condition at $x=0$ which can be written as
\begin{align}
p^{el}(x,s;y,0) = & \frac{e^{-\frac{(x-y)^2}{2t}}}{\sqrt{2\pi t}} - \frac{e^{-\frac{(x+y)^2}{2t}}}{\sqrt{2\pi t}} + 2e^{\alpha(x+y)} \int_{x+y}^{\infty} we^{-\alpha w} \frac{e^{-\frac{w^2}{2t}}}{\sqrt{2\pi t^3}}dw \label{sel1} \\
= & \frac{e^{-\frac{(x-y)^2}{2t}}}{\sqrt{2\pi t}} + \frac{e^{-\frac{(x+y)^2}{2t}}}{\sqrt{2\pi t}} - 2\alpha e^{\alpha(x+y)} \int_{x+y}^{\infty} e^{-\alpha w} \frac{e^{-\frac{w^2}{2t}}}{\sqrt{2\pi t}}dw\label{sel2}
\end{align}
By considering \eqref{sel1} and \eqref{sel2} we arrive at the explicit solution of problem \eqref{PDEelastic} which coincides with \eqref{eqelastic1} (or alternatively with \eqref{eqelastic2}).
\end{proof}

\begin{os}
\normalfont
For $\alpha \to \infty$ from  \eqref{sel1} we extract the absorbing solution \eqref{PDEabsorbing}. For $\alpha=0$ we obtain the reflecting solution, that is the reflecting solution to the equation of vibrations of rods with the Neumann condition at $x=0$. We can extract from \eqref{sel1} the second expression \eqref{sel2} by means of an integration by parts. In order to check that the solutions \eqref{eqelastic1} and \eqref{eqelastic2} satisfy the equation of vibrations of rods it is convenient to rewrite the third term of \eqref{eqelastic2} as
\[ 2\alpha \int_0^\infty  \frac{e^{-\alpha w}}{\sqrt{2\pi t}} \cos\left( \frac{((x+y) + w)^2}{2t} - \frac{\pi}{4} \right) dw. \]
\end{os}

\begin{os}
\normalfont
We note that the boundary conditions 
\[ u(0,t) = 0 \]
and
\begin{equation}
\frac{\partial^2 u}{\partial x^2}(x,t) \Bigg|_{x=0} = 0 \label{tmp9i}
\end{equation}
lead to the solution \eqref{urod2}. The condition \eqref{tmp9i} means that the bending moment at the free end point $x=0$ vanishes. The ''reflecting conditions'' 
\[ \frac{\partial u}{\partial x}(x,t) \Big|_{x=0} = 0 \]
and
\[ \frac{\partial^3 u}{\partial x^3}(x,t) \Big|_{x=0}=0 \]
both lead to the solution 
\[ u^{r}(x,t;y,0) = \frac{1}{\sqrt{2\pi t}} \left[ \cos\left( \frac{(x-y)^2}{2t} - \frac{\pi}{4} \right) + \cos\left( \frac{(x+y)^2}{2t}  - \frac{\pi}{4} \right) \right]. \]
\end{os}

\begin{os}
\normalfont
We now evaluate the ''survival measures'' of the elastic, reflecting and absorbing kernels obtained so far. We restrict ourselves to the elastic case since the other two follow without effort.
\begin{align}
\int_0^{\infty} u^{el}(x,t;y,0)dx = & \frac{1}{\sqrt{2\pi t}} \Bigg[ \int_0^\infty \bigg[ \cos\left( \frac{(x-y)^2}{2t} - \frac{\pi}{4} \right) + \cos\left( \frac{(x+y)^2}{2t} - \frac{\pi}{4} \right) \bigg]dx \nonumber \\
&  - 2\alpha \int_0^\infty e^{\alpha(x+y)} dx \int_{x+y}^{\infty} e^{-\alpha w} \cos\left( \frac{w^2}{2t} - \frac{\pi}{4} \right) \, dw \Bigg]\nonumber \\
= & \frac{1}{\sqrt{2\pi t}} \Bigg[ \int_{-y}^{\infty} \cos\left( \frac{w^2}{2t} - \frac{\pi}{4} \right) dw + \int_{y}^{\infty} \cos\left( \frac{w^2}{2t} - \frac{\pi}{4} \right)dw \nonumber \\
& - 2\alpha \int_{y}^{\infty} e^{-\alpha w} \cos\left( \frac{ w^2}{2t} - \frac{\pi}{4} \right)dw \int_0^{w-y} e^{\alpha(x+y)} dx \Bigg]\nonumber \\
= & \frac{1}{\sqrt{2\pi t}} \Bigg[ \int_{-y}^{y} \cos\left( \frac{w^2}{2t} - \frac{\pi}{4} \right) dw + 2 e^{\alpha y} \int_{y}^{\infty} e^{\alpha w} \cos\left( \frac{w^2}{2t} - \frac{\pi}{4} \right) dw \Bigg]. \label{secTin}
\end{align}
For $\alpha =0$ we have the reflecting case and the integral above equals one while for $\alpha \to \infty$ the second term in \eqref{secTin} goes to zero and we retrieve the ''survival measures'' of the absorbing case.
\end{os}

For rods of finite length we can get explicit solutions by resorting again to the decomposition of equation \eqref{eq1} on considering the solution of the corresponding Schr\"odinger equation. We have that
\begin{te}
For the rod of finite length $L$ we have that the solution of
\begin{equation}
\left\lbrace \begin{array}{l} \frac{\partial^2 u}{\partial t^2}(x,t) = -\frac{1}{2^2} \frac{\partial^4 u}{\partial x^4}(x,t), \quad 0 < x < L, \; t>0\\
u(x,0) = \delta(x-y)\\
u_t(x,0)=0\\
u_x |_{x=0} = u_x |_{x=L} = 0 \end{array} \right .
\label{PDEfiniteL}
\end{equation}
is given by
\begin{align}
u^{r}(x,t;y,0) = & \frac{1}{\sqrt{2\pi t}} \sum_{k=-\infty}^{+\infty} \Bigg[ \cos\left( \frac{(x-y+2kL)^2}{2t} - \frac{\pi}{4}\right)  + \cos\left( \frac{(x+y+2kL)^2}{2t} - \frac{\pi}{4} \right) \Bigg].
\end{align}
\end{te}
\begin{proof}
The solution can be written as
\[ u^r(x,t;y,0) = \frac{1}{2}\left[ u(x, it;y,0) + u(x,-it; y,0) \right] \]
where $u(x,t;y,0)$ is the well-known solution of the Cauchy problem of the heat equation
\begin{equation}
\left\lbrace \begin{array}{l} \frac{\partial u}{\partial t}(x,t) = \frac{1}{2} \frac{\partial^2 u}{\partial x^2}(x,t), \quad 0 < x < L, \; t>0\\
u(x,0) = \delta(x-y)\\
u_t(x,0)=0\\
u_x |_{x=0} = u_x |_{x=L} = 0. \end{array} \right .
\end{equation}
\end{proof}

\section{Fractional equations of vibrations of rods and fractional diffusions}
\subsection{General results}
In this section we consider the fractional version of the equation of vibrations of rods
\begin{equation} 
\frac{\partial^{2\nu} u}{\partial t^{2\nu}} = - \frac{1}{2^2} \frac{\partial^{4} u}{\partial x^4}, \quad 0<\nu \leq 1, \; x \in \mathbb{R},\, t>0 \label{pdfgh}
\end{equation}
subject to the initial conditions
\begin{equation} 
u(x,0) = \delta(x), \quad 0 < \nu \leq 1  \label{init1L}
\end{equation}
and also
\begin{equation} 
u_t(x, 0) = 0 \quad \textrm{for} \quad 1/2 < \nu \leq 1 \label{init2L}
\end{equation}
The differential operator appearing in \eqref{pdfgh} must be understood in the sense of Dzhrbashyan-Caputo, that is (for information on this derivative consult Podlubny \cite{Pob99})
\[ \frac{\partial^\nu u}{\partial t^{\nu}}(x,t) = \frac{1}{\Gamma(m-\nu)} \int_0^t \frac{\frac{\partial^m u}{\partial t^m}(x,s)}{(t-s)^{\nu +1-m}}ds, \quad \textrm{ for } m-1 < \nu < m, \quad m \in \mathbb{N}. \]
For $\nu=1$ equation \eqref{pdfgh} coincides with the equation of vibration of rods and beams while for $\nu=1/2$ it gives the biquadratic heat equation. The latter equation has been investigated by many researchers since the Sixties and even a stochastic calculus related to it has been worked out (see for example Krylov \cite{Kry60}, Hochberg \cite{Hoc78}, Nikitin and Orsingher \cite{NO00}, Lachal \cite{Lach03, Lach07}). 

Our first result concerns the Fourier transform of the solution to \eqref{pdfgh} (with initial conditions \eqref{init1L} and \eqref{init2L}).
\begin{te}
The Fourier transform
\begin{equation}
U_{\nu}(\xi, t) = \int_{-\infty}^{+\infty} e^{i \beta x} u_{\nu}(x,t) dx
\end{equation}
of the solution to the Cauchy problem
\begin{equation}
\left\lbrace \begin{array}{l} \frac{\partial^{2\nu} u}{\partial t^{2\nu}} = - \frac{1}{2^2} \frac{\partial^4 u}{\partial x^4}, \quad 0<\nu \leq 1, \; x \in \mathbb{R},\, t>0 \\
u(x,0) = \delta(x)\\
u_t(x,0)= 0 \end{array} \right .
\label{freneqfrac2}
\end{equation}
reads
\begin{align}
U_{\nu}(\xi, t) = & \frac{1}{2} \left[ E_{\nu, 1}\left( i \frac{\beta^2 t^{\nu}}{2} \right) + E_{\nu, 1}\left( -i \frac{\beta^2t^{\nu}}{2}\right) \right] = E_{2\nu, 1}\left( \frac{\beta^4 t^{2\nu}}{2^2} \right)
\label{frenfracF}
\end{align}
where
\begin{equation}
E_{\nu,1}(z)=\sum_{k \geq 0} \frac{z^k}{\Gamma(\nu k +1)}
\end{equation}
is the Mittag-Leffler function.
\end{te}
\begin{proof}
The Laplace transform of the solution to \eqref{freneqfrac2} reads
\begin{align}
\int_0^\infty e^{-\mu t} \int_{-\infty}^{+\infty} e^{i\beta x} u_{2\nu}(x,t)\, dx\, dt = & \frac{1}{2} \left[ \frac{\mu^{\nu -1}}{\mu^\nu + i\beta^2/2} + \frac{\mu^{\nu -1}}{\mu^\nu - i \beta^2/2} \right] = \frac{\mu^{2\nu -1}}{\mu^{2\nu} + \beta^4/2^2}. \label{asdFr}
\end{align}
The inverse Laplace transform of \eqref{asdFr} yields \eqref{frenfracF}.
\end{proof}

For the evaluation of the inverse Fourier transform the first expression of \eqref{frenfracF} is particularly convenient. This is because each term in \eqref{frenfracF} represents the Fourier transform of the solution to the fractional diffusion equation
\begin{equation}
\left\lbrace \begin{array}{l} \frac{\partial^\nu u}{\partial t^\nu} = \lambda^2 \frac{\partial^2 u}{\partial x^2}, \quad 0<\nu \leq 1, \; x \in \mathbb{R},\, t>0\\
u(x,0)=\delta(x)\\
u_t(x,0)=0. \end{array} \right .
\label{eqORS}
\end{equation}
The solution to \eqref{eqORS} can be written down in terms of Wright functions
\[ W_{\alpha, \beta}(z) = \sum_{k \geq 0} \frac{z^k}{k! \Gamma(\alpha k + \beta)}, \quad \alpha>-1,\, \beta >0. \]
We now arrive at the next result, concerning the inverse Fourier transform of \eqref{frenfracF}.
\begin{te}
The solution to \eqref{eqORS} is given by
\begin{align}
u_{2\nu}(x,t) = & \frac{1}{\pi \sqrt{2t^{\nu}}} \sum_{m=0}^\infty \frac{1}{m!}\left(- \sqrt{2}\frac{|x|}{\sqrt{t^\nu}} \right)^{m}\cos \left( \frac{m+1}{4}\pi \right) \sin\left(\frac{m+1}{2}\pi \nu \right)\Gamma\left( \frac{m+1}{2}\nu \right) \label{sumsol}\\
= & \frac{1}{\sqrt{2t^{\nu}}} \sum_{m=0}^\infty \frac{1}{m!}\left(- \sqrt{2}\frac{|x|}{\sqrt{t^\nu}} \right)^{m} \frac{\cos \left( \frac{m+1}{4}\pi \right)}{ \Gamma\left( 1- \frac{m+1}{2}\nu \right)}
\end{align}
for $x \in \mathbb{R}$, $t>0$ and $0 < \nu \leq 1$.
\end{te}
\begin{proof}
For the fractional diffusion equation \eqref{eqORS} the solution reads
\begin{equation}
v_{\nu}(x,t)= \frac{1}{2\lambda \sqrt{t^\nu}} W_{-\frac{\nu}{2}, 1- \frac{\nu}{2}}\left( - \frac{|x|}{\lambda \sqrt{t^\nu}} \right)
\end{equation}
and has Fourier transform
\[ \int_{\mathbb{R}} e^{i \beta x} v_{\nu}(x, t) dx = E_{\nu, 1}\left( - \lambda^2 \beta^2 t^{\nu} \right). \]
Therefore, in view of \eqref{frenfracF} we have that
\begin{align}
u_{2\nu}(x,t) = & \frac{1}{\sqrt{2^3t^{\nu}}} \left[ \frac{1}{\sqrt{-i}}W_{-\frac{\nu}{2}, 1-\frac{\nu}{2}}\left( - \sqrt{-\frac{2}{i}} \frac{|x|}{\sqrt{t^\nu}} \right) + \frac{1}{\sqrt{i}} W_{-\frac{\nu}{2}, 1- \frac{\nu}{2}}\left( - \sqrt{\frac{2}{i}} \frac{|x|}{\sqrt{t^\nu}} \right) \right] \nonumber \\
= & \frac{1}{\sqrt{2^3t^{\nu}}} \left[ e^{i \frac{\pi}{4}} W_{-\frac{\nu}{2}, 1- \frac{\nu}{2}}\left( - \sqrt{2} \frac{|x|}{\sqrt{t^\nu}} e^{i \frac{\pi}{2}} \right)  + e^{-i \frac{\pi}{4}} W_{-\frac{\nu}{2}, 1- \frac{\nu}{2}}\left( - \sqrt{2} \frac{|x|}{\sqrt{t^\nu}} e^{-i\frac{\pi}{4}} \right) \right] \nonumber \\
= & \frac{1}{\sqrt{2^3t^{\nu}}} \sum_{m=0}^\infty \left[ e^{i \frac{\pi}{4}} \left( - \sqrt{2} \frac{|x| e^{i \frac{\pi}{4}}}{\sqrt{t^\nu}} \right)^{m} + e^{-i \frac{\pi}{4}} \left(- \sqrt{2}\frac{|x| e^{-i \frac{\pi}{4}}}{\sqrt{t^\nu}} \right)^{m} \right] \frac{1}{m! \Gamma\left(-\frac{\nu m}{2} + 1 - \frac{\nu}{2} \right)}\nonumber \\
= & \frac{1}{\sqrt{2^3t^{\nu}}} \sum_{m=0}^\infty \left(- \sqrt{2}\frac{|x|}{\sqrt{t^\nu}} \right)^{m}  \frac{e^{i \frac{\pi}{4} + i \frac{\pi}{4}m} + e^{-i \frac{\pi}{4} - i \frac{\pi}{4}m}}{m! \Gamma\left(-\frac{\nu m}{2} + 1 - \frac{\nu}{2} \right)}\nonumber\\
= & \frac{1}{\sqrt{2t^{\nu}}} \sum_{m=0}^\infty \left(- \sqrt{2}\frac{|x|}{\sqrt{t^\nu}} \right)^{m}  \frac{\cos \left( \frac{m+1}{4}\pi \right)}{m! \Gamma\left(-\frac{\nu m}{2} + 1 - \frac{\nu}{2} \right)}\nonumber\\
= & \frac{1}{\pi \sqrt{2t^{\nu}}} \sum_{m=0}^\infty \frac{1}{m!}\left(- \sqrt{2}\frac{|x|}{\sqrt{t^\nu}} \right)^{m}\cos \left( \frac{m+1}{4}\pi \right) \sin\left(\frac{m+1}{2}\pi \nu \right)\Gamma\left( \frac{m+1}{2}\nu \right). 
\end{align}
In the last step we applied the reflection formula
\[ \Gamma(z)\Gamma(1-z) = \frac{\pi}{\sin \pi z}. \]
\end{proof}

\begin{os}
\normalfont
We give now the Laplace transform of \eqref{sumsol}. Since 
\[ \int_0^\infty e^{-\mu t} t^{-\frac{\nu}{2}m - \frac{\nu}{2}} dt = \mu^{\frac{\nu}{2}m + \frac{\nu}{2} - 1} \Gamma\left(1-\frac{\nu}{2}m - \frac{\nu}{2} \right) \]
and
\[ \Gamma\left(1-\frac{\nu}{2}m - \frac{\nu}{2} \right) \Gamma\left(\frac{\nu}{2}m + \frac{\nu}{2} \right) = \frac{\pi}{\sin \frac{m+1}{2}\nu} \]
we extract from \eqref{sumsol} that
\begin{align}
\int_0^\infty e^{-\mu t} u_{2\nu}(x,t) dt = & \frac{1}{\sqrt{2}} \sum_{m=0}^{\infty} \frac{1}{m!}\left( -\sqrt{2} |x| \right)^m \mu^{\frac{\nu}{2}m + \frac{\nu}{2} - 1} \cos\left( \frac{m+1}{2}\pi \right)\nonumber \\
= & \frac{\mu^{\nu/2 -1}}{2\sqrt{2}} \sum_{m=0}^{\infty}  \frac{1}{m!}\left( -\sqrt{2} |x| \mu^{\nu/2} \right)^m \left( e^{i \frac{\pi}{4} + i \frac{m\pi}{4}} + e^{-i \frac{\pi}{4} - i \frac{m\pi}{4}} \right)\nonumber \\
= & \frac{\mu^{\nu/2 -1}}{2\sqrt{2}} \left( e^{i \frac{\pi}{4} - \sqrt{2} |x| e^{i\frac{\pi}{4}} \mu^{\nu/2}} + e^{-i \frac{\pi}{4} - \sqrt{2} |x| e^{-i\frac{\pi}{4}} \mu^{\nu/2}} \right)\nonumber \\
= & \frac{\mu^{\nu/2 -1} e^{-|x| \mu^{\nu/2}}}{2\sqrt{2}} \left( e^{i \frac{\pi}{4} - i |x| \mu^{\nu/2}} + e^{-i \frac{\pi}{4} -i |x| \mu^{\nu/2}} \right)\nonumber \\
= & \frac{\mu^{\nu/2 -1} e^{-|x| \mu^{\nu/2}}}{\sqrt{2}} \cos\left( |x| \mu^{\nu/2} - \frac{\pi}{4} \right). \label{fghR}
\end{align}
We now take the Fourier transform of \eqref{fghR} 
\begin{align*}
& \frac{\mu^{\frac{\nu}{2} - 1}}{2\sqrt{2}} \left[ e^{i \frac{\pi}{4}} \int_{+\infty}^{-\infty} e^{i\beta x} e^{-|x| \mu^{\nu/2} e^{i\frac{\pi}{4}} \sqrt{2}} + e^{-i \frac{\pi}{4}} \int_{+\infty}^{-\infty} e^{i\beta x} e^{-|x| \mu^{\nu/2} e^{-i\frac{\pi}{4}} \sqrt{2}} \right]\\
= & \mu^{\frac{\nu}{2} - 1} \left[ \frac{e^{i\frac{\pi}{2}} \mu^{\nu/2}}{\beta^2 + 2\mu^{\nu} e^{i \frac{\pi}{2}}} + \frac{e^{-i\frac{\pi}{2}} \mu^{\nu/2}}{\beta^2 + 2 \mu^{\nu} e^{-i\frac{\pi}{2}}} \right]\\
= & \mu^{\frac{\nu}{2} -1} \left[ \frac{i(\beta^2 -2 i \mu^\nu ) -i (\beta^2 +2i \mu^\nu)}{2^2 \mu^{2\nu} + \beta^4} \right] 
= \frac{\mu^{2\nu -1}}{\mu^{2\nu}+ \frac{\beta^4}{2^2}}
\end{align*}
and this confirms result \eqref{asdFr}. For a further check we evaluate the inverse Fourier transform of \eqref{asdFr}.
\begin{align*}
\frac{\mu^{2\nu -1}}{2\pi} \int_{-\infty}^{+\infty} \frac{e^{-i\beta x}\, d\beta}{\mu^{2\nu} + \frac{\beta^2}{2^2}} = & \frac{\mu^{\nu -1}}{2\pi} \left[ \int_{-\infty}^{+\infty} \frac{e^{-i\beta x}\, d\beta}{\mu^{\nu} + i \frac{\beta^2}{2}} + \int_{-\infty}^{+\infty} \frac{e^{-i\beta x}\, d\beta}{\mu^{\nu} - i \frac{\beta^2}{2}} \right]\\
= & \frac{\mu^{\nu -1}}{2^2 \pi} \Bigg[ \int_{-\infty}^{+\infty} e^{-i \beta x} \frac{\left( 2 e^{-i \frac{\pi}{2} \mu^{\nu}} \right)^{1/2}}{\left( 2 e^{-i \frac{\pi}{2} \mu^{\nu}} \right) + \beta^2} \frac{2}{e^{i \frac{\pi}{2}}} \frac{d\beta}{(2 e^{-i \frac{\pi}{2} \mu^{\nu}})^{1/2}} \\
 & +  \int_{-\infty}^{+\infty} e^{-i \beta x} \frac{\left( 2 e^{i \frac{\pi}{2} \mu^{\nu}} \right)^{1/2}}{\left( 2 e^{i \frac{\pi}{2} \mu^{\nu}} \right) + \beta^2} \frac{2}{e^{-i \frac{\pi}{2}}} \frac{d\beta}{(2 e^{i \frac{\pi}{2} \mu^{\nu}})^{1/2}} \Bigg]\\
= & \frac{\mu^{\nu -1}}{2 \sqrt{2}} \Bigg[ \frac{e^{-|x| (2e^{i\frac{\pi}{2}} \mu^{\nu})^{1/2} + i\frac{\pi}{2}}}{(2 e^{i \frac{\pi}{2}} \mu^{\nu})^{1/2}} + \frac{e^{-|x| (2e^{-i\frac{\pi}{2}} \mu^{\nu})^{1/2} - i\frac{\pi}{2}}}{(2 e^{-i \frac{\pi}{2}} \mu^{\nu})^{1/2}} \Bigg]\\
= & \frac{\mu^{\frac{\nu}{2} - 1}}{2 \sqrt{2}} \left[ e^{-|x| (1-i) \mu^{\nu/2} - i \frac{\pi}{4}} + e^{-|x|(1+i)\mu^{\nu/2} + i \frac{\pi}{4}} \right]\\
= & \frac{\mu^{\frac{\nu}{2} - 1} e^{-|x| \mu^{\nu/2}}}{\sqrt{2}} \frac{1}{2} \left[ e^{i|x|\mu^{\nu/2}-i \frac{\pi}{4}} + e^{-i |x| \mu^{\nu/2} + i \frac{\pi}{4}}  \right]\\
= & \frac{\mu^{\frac{\nu}{2} - 1} e^{-|x| \mu^{\nu/2}}}{\sqrt{2}} \cos \left(|x| \mu^{\nu/2} - \frac{\pi}{4} \right).
\end{align*}
\end{os}

\begin{te}
The solution to the fractional equation of vibrations of rods
\begin{equation}
\left\lbrace \begin{array}{l}
\frac{\partial^{2\nu} u}{\partial t^{2\nu}} = - \frac{\lambda^2}{2^2} \frac{\partial^4 u}{\partial x^4}, \quad x \in \mathbb{R}, \, t>0\\
u(x,0) = \delta(x), \quad 0 < \nu \leq 1\\
u_t(x,0)=0, \quad 1/2 < \nu \leq 1 \end{array} \right .
\label{pdetretre}
\end{equation} 
coincides with 
\begin{equation}
u_{2\nu}(x,t) = \int_0^\infty \frac{1}{\sqrt{2\pi s}} \cos \left( \frac{x^2}{2s} - \frac{\pi}{4} \right) v_{2\nu}(s, t)\, ds \label{frau2ni}
\end{equation}
where $v_{2\nu}$ is the folded solution of the fractional diffusion equation
\begin{equation}
\frac{\partial^{2\nu} u}{\partial t^{2\nu}}= \lambda^2 \frac{\partial^2 u}{\partial x^2} \label{fracdifeqP}
\end{equation}
subject to the initial conditions 
\begin{equation}
\begin{array}{l} u(x,0)=\delta(x), \quad 0 < \nu \leq 1,\\ u_t(x,0)=0, \quad 1/2 < \nu \leq 1. \end{array} \label{pdev}
\end{equation}
\label{tretre}
\end{te}
\begin{proof}
The fractional derivative of order $2\nu$ of \eqref{frau2ni} yields
\begin{align*}
\frac{\partial^{2\nu} u_{2\nu}}{\partial t^{2\nu}} = & \int_0^\infty \frac{1}{\sqrt{2\pi s}} \cos \left( \frac{x^2}{2s} - \frac{\pi}{4} \right) \frac{\partial^{2\nu} v_{2\nu}}{\partial t^{2\nu}}(s, t)\, ds\\
= & [\textrm{by } \eqref{pdev}] =  \lambda^2 \int_0^\infty \frac{1}{\sqrt{2\pi s}} \cos \left( \frac{x^2}{2s} - \frac{\pi}{4} \right) \frac{\partial^{2} v_{2\nu}}{\partial s^2}(s, t)\, ds\\
= & \lambda^2 \int_0^\infty \frac{\partial^{2} }{\partial s^2} \left[ \frac{1}{\sqrt{2\pi s}} \cos \left( \frac{x^2}{2s} - \frac{\pi}{4} \right)\right] v_{2\nu}(s, t)\, ds = [\textrm{by } \eqref{eq1}] = - \frac{\lambda^2}{2^2} \frac{\partial^4 u_{2\nu}}{\partial x^4}.
\end{align*}
We note that
\begin{align*}
\int_{\mathbb{R}} e^{i \beta x} u_{2\nu}(x,t) dx = & \left[ \textrm{by } \eqref{frau2ni} \right] = \int_0^\infty \cos \frac{\beta^2 s}{2}\, v_{2\nu}(s, t)\, ds\\
= & \int_{\mathbb{R}} e^{i \frac{\beta^2 s}{2}} v_{2\nu}(s, t)\, ds =  E_{2\nu, 1}\left( \frac{\lambda^2 \beta^4}{2^2}t \right).
\end{align*}
\end{proof}

\begin{os}
\normalfont
Theorem \ref{tretre} is the counterpart of result (2.2) for fractional diffusions treated in \cite{OB09} where the role of the Gaussian density is here played by the Fresnel function \eqref{funQW}. In some cases this is extremely fruitful because it permits writing down the fundamental solution explicitely. For $\nu=1/3$, from \eqref{pdetretre} we obtain that 
\begin{align}
u_{2/3}(x,t) = & \int_0^\infty \frac{1}{\sqrt{2 \pi s}} \cos \left( \frac{x^2}{2s} - \frac{\pi}{4} \right) v_{2/3}(s, t)ds \label{intermstep}\\
= & \int_0^\infty \frac{1}{\sqrt{2 \pi s}} \cos \left( \frac{x^2}{2s} - \frac{\pi}{4} \right) \frac{3}{\lambda} \frac{1}{\sqrt[3]{3t}} Ai\left( \frac{s}{\lambda \sqrt[3]{3t}} \right) \, ds. \nonumber
\end{align} 
In the intermediate step of \eqref{intermstep} we employ the result
\begin{equation}
\left\lbrace \begin{array}{l} \frac{\partial^{2/3} u}{\partial t^{2/3}} = \lambda^2 \frac{\partial^2 u}{\partial x^2}, \quad x \in \mathbb{R},\, t>0\\ u(x,0) = \delta(x) \end{array} \right .
\end{equation}
whose solution is (see \cite[formula (4.2)]{OB09})
\[v_{2/3}(x,t) = \frac{3}{2 \lambda} \frac{1}{\sqrt[3]{3t}} Ai\left( \frac{|x|}{\lambda \sqrt[3]{3t}} \right), \quad x \in \mathbb{R},\, t>0. \]
Formula \eqref{intermstep} is an integral version of \eqref{sumsol} for $\nu=1/3$. The relationship \eqref{frau2ni} relates the solutions of the fractional equation  of rods \eqref{pdetretre} with the fractional diffusion equation \eqref{fracdifeqP} of which a vast literature exists.
\end{os}

\begin{figure}[ht]
\centering
\includegraphics[scale=0.5]{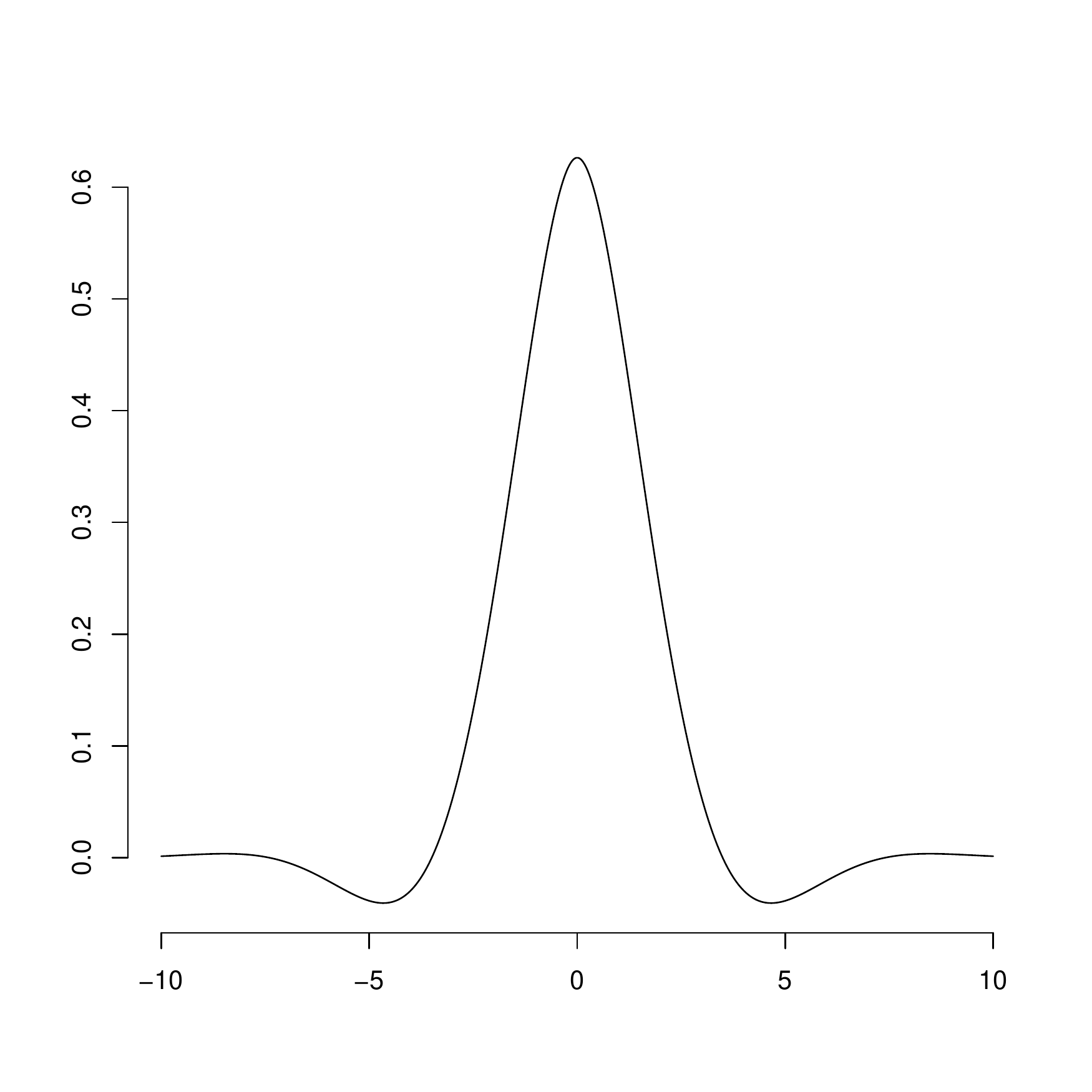}
\caption{The fundamental solution $u_{2/3}$ of the time-fractional Fresnel equation. }
\end{figure}

\subsection{Special cases}
We now examine some special cases.
\begin{os}
\normalfont
For $\nu=1$, we are able to extract from \eqref{sumsol} the fundamental solution \eqref{funQW}. For $n \in \mathbb{N}$, we consider in \eqref{sumsol} the cases $m=2n+1$ (for which the series is equal to zero) and $m=2n$ for which we have that
\begin{align*}
u_{2}(x,t) = & \frac{1}{\pi \sqrt{2t}} \sum_{m=0}^\infty \frac{1}{m!}\left(- \sqrt{2}\frac{|x|}{t^{1/2}} \right)^{m}\cos \left( \frac{m+1}{4}\pi \right) \sin\left(\frac{m+1}{2} \pi \right)\Gamma\left( \frac{m+1}{2} \right)\\
= & \frac{1}{\pi \sqrt{2t}} \sum_{n=0}^\infty \frac{1}{(2n)!}\left(2 \frac{|x|^2}{t} \right)^{n}\cos \left( \frac{2n+1}{4}\pi \right) \sin\left(\frac{2n+1}{2} \pi \right)\Gamma\left( \frac{2n+1}{2} \right)\\
= & \frac{1}{\pi \sqrt{2t}} \sum_{n=0}^\infty \frac{(-1)^n}{(2n)!}\left(2 \frac{|x|^2}{t} \right)^{n}\cos \left( \frac{2n+1}{4}\pi \right)  \Gamma\left( n + \frac{1}{2} \right)\\
= & \frac{2}{\sqrt{2\pi t}} \sum_{n=0}^\infty \frac{(-1)^n}{(2n)!}\left(\frac{|x|^2}{2t} \right)^{n}\cos \left( \frac{n}{2}\pi + \frac{\pi}{4} \right)  \frac{\Gamma(2n)}{\Gamma(n)}\\
= & \frac{1}{\sqrt{2\pi t}} \sum_{n=0}^\infty \frac{(-1)^n}{n!}\left(\frac{|x|^2}{2t} \right)^{n} \frac{1}{2} \left( e^{i\frac{n}{2}\pi + i\frac{\pi}{4}} + e^{-i\frac{n}{2}\pi - i\frac{\pi}{4}} \right) \\
= & \frac{1}{2\sqrt{2\pi t}} \sum_{n=0}^\infty \frac{1}{n!} \left[ \left(- i \frac{|x|^2}{2t} \right)^{n} e^{i\frac{\pi}{4}} + \left(i \frac{|x|^2}{2t} \right)^{n} e^{- i\frac{\pi}{4}} \right] \\
= & \frac{1}{2\sqrt{2\pi t}} \left( e^{- i \frac{|x|^2}{2t} + i\frac{\pi}{4}} + e^{+ i \frac{|x|^2}{2t} - i\frac{\pi}{4}} \right) =  \frac{1}{\sqrt{2\pi t}} \cos\left( \frac{x^2}{2t} - \frac{\pi}{4} \right).
\end{align*}
\end{os}

\begin{os}
\normalfont
Another important case is $\nu=1/2$ and in this case we easily arrive at
\begin{align}
u_1(x,t) = & \frac{1}{2\pi} \int_{-\infty}^{+\infty} e^{-\frac{y^4 t}{2^2}} \cos\left(xy \right)dy \label{int4}\\
= & \frac{1}{2\pi \sqrt{2t^{1/2}}} \sum_{k=0}^{\infty} \frac{1}{(2k)!} \left( - \frac{\sqrt{2}|x|}{t^{1/4}} \right)^{2k} (-1)^k \Gamma\left( \frac{k}{2} + \frac{1}{4} \right). \nonumber
\end{align}
The integral \eqref{int4} had been studied by Bernstein \cite{Bernst19} and its connection with the fundamental solution of the biquadratic heat equation is well-known. By applying the multiplication formula of the Gamma function
\[ \Gamma(z)\Gamma\left( z + \frac{1}{m} \right)\Gamma\left( z + \frac{2}{m} \right) \ldots \Gamma\left( z + \frac{m - 1}{m} \right) = (2\pi)^{\frac{m-1}{2}} m^{\frac{1}{2} - m z} \Gamma(m z). \]
for $m=4$ and $z=k/2$, we have that
\begin{align*}
\frac{\Gamma\left( \frac{k}{2} + \frac{1}{4} \right)}{\Gamma\left( 2k + 1 \right)} = & \frac{\pi 2^{\frac{1}{2} - 3k}}{k!\, \Gamma\left( \frac{k}{2} + \frac{3}{4} \right)}.
\end{align*}
This leads to the alternative form
\begin{equation}
u_1(x,t) = \frac{1}{2 \sqrt[4]{t}} \sum_{k=0}^{\infty} \left( \frac{|x|}{2 \sqrt[4]{t}} \right)^{2k} \frac{(-1)^k}{k!\, \Gamma\left( \frac{k}{2} + \frac{3}{4} \right)}.
\end{equation}
\end{os}

Particular attention is drawn to the case $\nu=2/3$ for which we produce different alternative expressions of the solution in terms of Airy functions.

\begin{te}
The solution to
\begin{equation}
\frac{\partial^{4/3} u}{\partial t^{4/3}} = - \frac{1}{2^2} \frac{\partial^4 u}{\partial x^4}, \quad  x \in \mathbb{R},\, t>0
\end{equation}
with 
\begin{equation}
\left\lbrace \begin{array}{ll} u(x, 0) = \delta(x)\\ u_t(x, 0) = 0 \end{array} \right .
\end{equation}
has the form
\begin{equation}
u_{4/3}(x,t)=  \frac{3}{2^{3/2} \sqrt[3]{3t}} \left[ e^{i \frac{\pi}{4}} Ai\left( \sqrt{2} \frac{|x| e^{i \frac{\pi}{4}}}{\sqrt[3]{3t}} \right) + e^{-i \frac{\pi}{4}} Ai\left( \sqrt{2} \frac{|x| e^{-i \frac{\pi}{4}}}{\sqrt[3]{3t}} \right)  \right]  \label{quattroterzi}
\end{equation}
where
\begin{equation}
Ai(x) = \frac{1}{\pi} \int_0^\infty \cos\left( \alpha x + \frac{\alpha^3}{3} \right) d\alpha = \frac{\sqrt{x}}{3} \left[ I_{-\frac{1}{3}}\left( \frac{2}{3} x^{3/2} \right) - I_{\frac{1}{3}}\left( \frac{2}{3} x^{3/2} \right) \right] \label{second427}
\end{equation}
is the Airy function (the second formula in \eqref{second427} holds for $x>0$).
\end{te}
\begin{proof}
For $\nu=2/3$, formula \eqref{sumsol} yields
\begin{equation}
u_{4/3}(x,t) = \frac{1}{\pi \sqrt{2t^{2/3}}} \sum_{m=0}^{\infty} \frac{ \left( -\sqrt{2} \frac{|x|}{\sqrt[3]{t}} \right)^{m}}{m!} \cos \left( \frac{m+1}{4}\pi \right) \sin \left( \frac{m+1}{3}\pi \right) \Gamma\left( \frac{m+1}{3} \right). \label{sumTMl}
\end{equation}
We split the series \eqref{sumTMl} into the following three cases:
\begin{equation}
\left\lbrace 
\begin{array}{l}
m=3k+2,\\
m=3k+1,\\
m=3k.
\end{array}
\right .
\end{equation}
For our convenience we write 
\[ u_{4/3}(x,t) = A_{2} + A_{1} + A_{0}. \]
For $m=2k+2$ the series $A_{2}$ equals to zero because 
\[ \sin\left( \frac{m+1}{3}\pi \right) = \sin \left( (k+1)\pi \right)=0 \]
for all integer values of $k$. For $n=3k+1$, formula \eqref{sumsol} takes the form: 
\begin{align}
& \frac{1}{\pi \sqrt{2} \sqrt[3]{t}} \sum_{k=0}^{\infty} \frac{1}{(3k+1)!} \left( -\sqrt{2} \frac{|x|}{\sqrt[3]{t}} \right)^{3k+1}  \cos\left( \frac{3k+2}{4}\pi \right) \sin\left( \frac{3k+2}{3}\pi \right) \Gamma\left( k + \frac{2}{3} \right)\nonumber \\
= &  \frac{1}{\pi \sqrt{2} \sqrt[3]{t}} \sum_{k=0}^{\infty} \frac{1}{(3k+1)!} \left( -\sqrt{2} \frac{|x|}{\sqrt[3]{t}} \right)^{3k+1} \left( - \sin \frac{3k}{4}\pi \right) (-1)^k \sin \frac{2}{3}\pi \, \Gamma\left( k + \frac{2}{3} \right). \label{serTrip}
\end{align}
In view of the triplication formula (see \cite[page 14]{LE})
\[ \Gamma(z) \Gamma\left( z+ \frac{1}{3} \right) \Gamma\left( z+ \frac{2}{3} \right) =  \frac{2\pi}{3^{3z-1/2}} \Gamma(3z) \]
we have that
\begin{equation}
\frac{\Gamma\left( k + \frac{2}{3} \right)}{\Gamma\left( 3k+2\right)} =  \frac{1}{3^2} \frac{2\pi}{3^{3k-1/2}} \frac{1}{\Gamma(k+1) \Gamma\left(k + \frac{1}{3} + 1\right)}. \label{trip}
\end{equation}
The series \eqref{serTrip}, in force of \eqref{trip}, becomes
\begin{align}
A_{1} = & \frac{2\pi \sin \frac{2\pi}{3}}{\pi \sqrt{2} \sqrt[3]{t}} \frac{1}{3^2} \sum_{k=0}^{\infty} \left( \sqrt{2} \frac{|x|}{\sqrt[3]{t}} \right)^{3k+1} \sin \left( \frac{3k}{4}\pi \right) \frac{1}{3^{3k -1/2}} \frac{1}{\Gamma(k+1) \Gamma\left( k+\frac{1}{3} + 1 \right)}\nonumber \\
= & \frac{1}{\sqrt{2} \sqrt[3]{t}} \sum_{k=0}^{\infty} \frac{1}{k!} \left( \sqrt{2} \frac{|x|}{3 \sqrt[3]{t}} \right)^{3k+1} \frac{1}{\Gamma\left( k + \frac{1}{3} + 1 \right)} \frac{1}{2i} \left[ e^{i \frac{3k}{4}\pi} - e^{-i \frac{3k}{4}\pi}  \right]\nonumber \\
= & \frac{1}{2i\,\sqrt{2} \sqrt[3]{t}} \sum_{k=0}^{\infty} \left[ e^{-i \frac{\pi}{4}}  \left(\sqrt{2} \frac{|x| e^{i\frac{\pi}{4}}}{3\sqrt[3]{t}} \right)^{3k+1} - e^{i \frac{\pi}{4}}  \left(\sqrt{2} \frac{|x| e^{-i\frac{\pi}{4}}}{3\sqrt[3]{t}} \right)^{3k+1} \right] \frac{1}{k! \Gamma\left( k + \frac{1}{3} +1 \right)} \nonumber \\
= & \frac{1}{2i\,\sqrt{2} \sqrt[3]{t}} \left[ e^{-i \frac{\pi}{4}} \left( \sqrt{2} \frac{|x| e^{i \frac{\pi}{4}}}{3\sqrt[3]{t}} \right)^{1/2} \sum_{k=0}^{\infty} \left(\left( \sqrt{2} \frac{|x| e^{i \frac{\pi}{4}}}{3\sqrt[3]{t}} \right)^{3/2} \right)^{2k+\frac{1}{3}} \frac{1}{k! \Gamma\left( k + \frac{1}{3} +1 \right)} \right .\nonumber \\
& \left .-  e^{i \frac{\pi}{4}}  \left( \sqrt{2} \frac{|x| e^{-i \frac{\pi}{4}}}{3\sqrt[3]{t}} \right)^{1/2} \sum_{k=0}^{\infty} \left( \left( \sqrt{2} \frac{|x| e^{-i \frac{\pi}{4}}}{3\sqrt[3]{t}} \right)^{3/2} \right)^{2k + \frac{1}{3}} \frac{1}{k! \Gamma\left( k + \frac{1}{3} +1 \right)} \right] \nonumber \\
= & \frac{1}{2i\,\sqrt{2} \sqrt[3]{t}} \left[ e^{-i \frac{\pi}{4}} \left( \sqrt{2} \frac{|x| e^{i \frac{\pi}{4}}}{3\sqrt[3]{t}} \right)^{1/2} I_{1/3}\left(2 \left( \sqrt{2} \frac{|x| e^{i \frac{\pi}{4}}}{3\sqrt[3]{t}} \right)^{3/2} \right)  \right .\nonumber \\
& \left .-  e^{i \frac{\pi}{4}} \left( \sqrt{2} \frac{|x| e^{-i \frac{\pi}{4}}}{3\sqrt[3]{t}} \right)^{1/2} I_{1/3}\left(2 \left( \sqrt{2} \frac{|x| e^{-i \frac{\pi}{4}}}{3\sqrt[3]{t}} \right)^{3/2} \right) \right] \nonumber \\
= & \frac{1}{2\,\sqrt{2} \sqrt[3]{t}} \left[ -e^{i \frac{\pi}{4}} \left( \sqrt{2} \frac{|x| e^{i \frac{\pi}{4}}}{3\sqrt[3]{t}} \right)^{1/2} I_{1/3}\left(2 \left( \sqrt{2} \frac{|x| e^{i \frac{\pi}{4}}}{3\sqrt[3]{t}} \right)^{3/2} \right)  \right .\nonumber \\
& \left .-  e^{-i \frac{\pi}{4}} \left( \sqrt{2} \frac{|x| e^{-i \frac{\pi}{4}}}{3\sqrt[3]{t}} \right)^{1/2} I_{1/3}\left(2 \left( \sqrt{2} \frac{|x| e^{-i \frac{\pi}{4}}}{3\sqrt[3]{t}} \right)^{3/2} \right) \right]. \label{secondcase}
\end{align}
With similar calculation we evaluate $A_{3k}$
\begin{align}
A_{0} = & \frac{1}{\pi \sqrt{2}\sqrt[3]{t}} \sum_{k=0}^{\infty} \frac{1}{(3k)!} \left( -\sqrt{2} \frac{|x|}{\sqrt[3]{t}} \right)^{3k} \cos\left( \frac{3k+1}{4}\pi \right) \sin \left( \frac{3k+1}{3}\pi \right) \Gamma\left( \frac{3k+1}{3} \right)\nonumber \\
= & \frac{3^{1/2}}{\pi 2^{3/2} \sqrt[3]{t}} \sum_{k=0}^{\infty} \frac{1}{(3k)!}\left( \sqrt{2} \frac{|x|}{\sqrt[3]{t}} \right)^{3k} \cos\left( \frac{3k+1}{4}\pi \right) \Gamma\left( k + \frac{1}{3} \right)\nonumber \\
= & \frac{3^{1/2}}{2^{1/2}\sqrt[3]{t}} \sum_{k=0}^{\infty} \frac{1}{k!} \left( \sqrt{2} \frac{|x|}{\sqrt[3]{t}} \right)^{3k}\cos\left( \frac{3k+1}{4}\pi \right) \frac{1}{3^{3k + 1/2} \Gamma\left( k - \frac{1}{3} + 1 \right)}\nonumber \\
= & \frac{1}{2^{3/2}\sqrt[3]{t}} \sum_{k=0}^{\infty} \left[ e^{i \frac{\pi}{4}} \left( \sqrt{2} \frac{|x| e^{i \frac{\pi}{4}}}{3 \sqrt[3]{t}} \right)^{3k} + e^{-i \frac{\pi}{4}} \left( \sqrt{2} \frac{|x| e^{-i \frac{\pi}{4}}}{3 \sqrt[3]{t}} \right)^{3k} \right] \frac{1}{k! \Gamma\left( k - \frac{1}{3} + 1 \right)}\nonumber \\
= & \frac{1}{2^{3/2}\sqrt[3]{t}} \left[ e^{i \frac{\pi}{4}} \left( \sqrt{2} \frac{|x| e^{i \frac{\pi}{4}}}{3 \sqrt[3]{t}} \right)^{1/2} \sum_{k=0}^{\infty} \left( \left( \sqrt{2} \frac{|x| e^{i \frac{\pi}{4}}}{3 \sqrt[3]{t}} \right)^{3/2} \right)^{2k - \frac{1}{3}} \frac{1}{k! \Gamma\left( k - \frac{1}{3} + 1 \right)}  \right .\nonumber \\
& + \left .e^{-i \frac{\pi}{4}} \left( \sqrt{2} \frac{|x| e^{-i \frac{\pi}{4}}}{3 \sqrt[3]{t}} \right)^{1/2} \sum_{k=0}^{\infty} \left( 2 \left( \sqrt{2} \frac{|x| e^{-i \frac{\pi}{4}}}{3 \sqrt[3]{t}} \right)^{3/2} \right)^{2k - \frac{1}{3}}\frac{1}{k! \Gamma\left( k - \frac{1}{3} + 1 \right)} \right] \nonumber \\
= & \frac{1}{2^{3/2}\sqrt[3]{t}} \left[ e^{i \frac{\pi}{4}} \left( \sqrt{2} \frac{|x| e^{i \frac{\pi}{4}}}{3 \sqrt[3]{t}} \right)^{1/2} I_{-1/3}\left( 2 \left( \sqrt{2} \frac{|x| e^{i \frac{\pi}{4}}}{3 \sqrt[3]{t}} \right)^{3/2} \right)  \right .\nonumber \\
& + \left .e^{-i \frac{\pi}{4}} \left( \sqrt{2} \frac{|x| e^{-i \frac{\pi}{4}}}{3 \sqrt[3]{t}} \right)^{1/2} I_{-1/3} \left( 2 \left( \sqrt{2} \frac{|x| e^{-i \frac{\pi}{4}}}{3 \sqrt[3]{t}} \right)^{3/2} \right) \right]. \label{thirdcase}
\end{align}
By summing up $A_{3k+1}$ and $A_{3k}$ we have 
\begin{align*}
u_{4/3}(x,t) =  & \frac{1}{2^{3/2}\sqrt[3]{t}} \left[ e^{i \frac{\pi}{4}} \left( \sqrt{2} \frac{|x| e^{i \frac{\pi}{4}}}{3 \sqrt[3]{t}} \right)^{1/2} I_{-1/3}\left( 2 \left( \sqrt{2} \frac{|x| e^{i \frac{\pi}{4}}}{3 \sqrt[3]{t}} \right)^{3/2} \right)  \right .\nonumber \\
& + e^{-i \frac{\pi}{4}} \left( \sqrt{2} \frac{|x| e^{-i \frac{\pi}{4}}}{3 \sqrt[3]{t}} \right)^{1/2} I_{-1/3} \left( 2 \left( \sqrt{2} \frac{|x| e^{-i \frac{\pi}{4}}}{3 \sqrt[3]{t}} \right)^{3/2} \right) \\
& -e^{i \frac{\pi}{4}} \left( \sqrt{2} \frac{|x| e^{i \frac{\pi}{4}}}{3\sqrt[3]{t}} \right)^{1/2} I_{1/3}\left(2 \left( \sqrt{2} \frac{|x| e^{i \frac{\pi}{4}}}{3\sqrt[3]{t}} \right)^{3/2} \right)  \nonumber \\
& \left .-  e^{-i \frac{\pi}{4}} \left( \sqrt{2} \frac{|x| e^{-i \frac{\pi}{4}}}{3\sqrt[3]{t}} \right)^{1/2} I_{1/3}\left(2 \left( \sqrt{2} \frac{|x| e^{-i \frac{\pi}{4}}}{3\sqrt[3]{t}} \right)^{3/2} \right) \right]\\
= & \frac{3}{2^{3/2} \sqrt[3]{3t}} \left[ e^{i \frac{\pi}{4}} Ai\left( \sqrt{2} \frac{|x| e^{i \frac{\pi}{4}}}{\sqrt[3]{3t}} \right) + e^{-i \frac{\pi}{4}} Ai\left( \sqrt{2} \frac{|x| e^{-i \frac{\pi}{4}}}{\sqrt[3]{3t}} \right)  \right].
\end{align*}
\end{proof}

\begin{os}
\normalfont
Result \eqref{quattroterzi} shows that the solution to
\begin{equation}
\frac{\partial^{4/3} u}{\partial t^{4/3}} = - \frac{1}{2^2} \frac{\partial^4 u}{\partial x^4} \label{relsolOrseq}
\end{equation}
can be expressed as
\begin{align}
u_{4/3}(x,t) = & \frac{1}{2}\left[ v_{2/3}(x, te^{i \frac{\pi}{4}}) + v_{2/3}(x, te^{-i \frac{\pi}{4}}) \right]\nonumber \\
= & \frac{3}{2^{2} \sqrt[3]{3t}} \left[ \frac{1}{\lambda_1} Ai\left(\frac{|x| }{\lambda_1 \sqrt[3]{3t}} \right) + \frac{1}{\lambda_2} Ai\left( \frac{|x|}{\lambda_2 \sqrt[3]{3t}} \right)  \right] \label{relsolOrs}
\end{align}
with
\[ \lambda_1 = \frac{e^{-i\frac{\pi}{4}}}{\sqrt{2}}= \sqrt{-\frac{i}{2}} \qquad \lambda_2 = \frac{e^{i \frac{\pi}{4}}}{\sqrt{2}} = \sqrt{\frac{i}{2}}\]
where
\begin{equation}
v_{2/3}(x, t) = \frac{3}{2} \frac{1}{\lambda \sqrt[3]{3t}}Ai\left( \frac{|x|}{\lambda \sqrt[3]{3t}} \right)
\end{equation}
is the solution to
\begin{equation}
\frac{\partial^{2/3} u}{\partial t^{2/3}} = \lambda^2 \frac{\partial^2 u}{\partial x^2}, \quad u(x,0)=\delta(x)
\end{equation}
(see formula (4.2) of \cite{OB09}). Formula \eqref{relsolOrs} shows that the solution to \eqref{relsolOrseq} can be viewed as the superposition of solutions of fractional diffusion equations of order $2/3$ with imaginary time running in opposite directions. Result \eqref{relsolOrs} can also be obtained by inverting the Fourier transform \eqref{frenfracF} for $\nu=2/3$. 
\end{os}

\begin{figure}
\centering
\includegraphics[scale=.5]{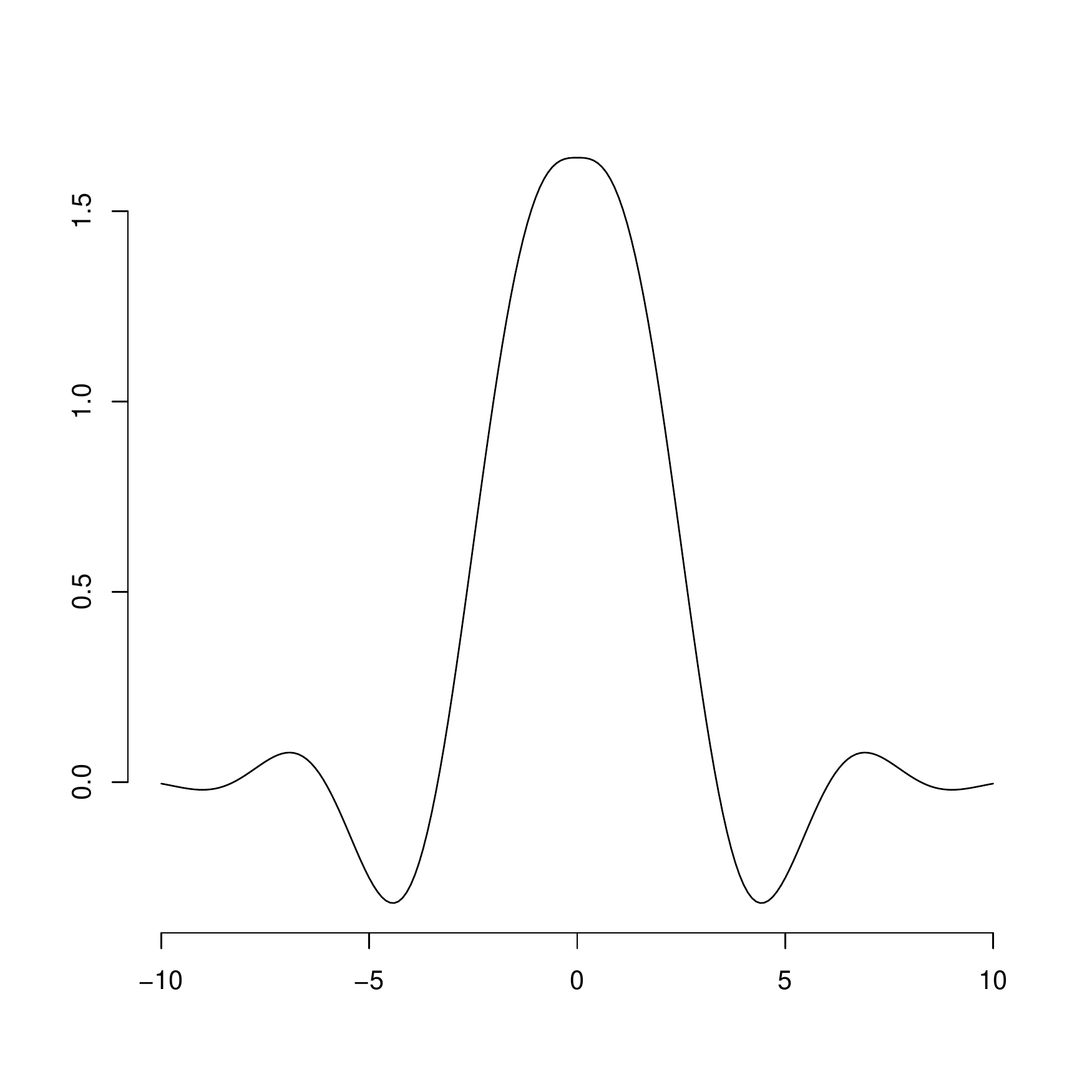}
\caption{$u_{4/3}$  :  the function \eqref{quattroterzi}}
\end{figure}

\begin{os}
\normalfont
By introducing the drift in the Schr\"odinger equations appearing in the equation of vibrations of rods we get that
\begin{align}
0 = & \left[\frac{\partial}{\partial t} - \frac{i}{2}\left( \frac{\partial^2}{\partial x^2} - \mu \frac{\partial}{\partial x} \right) \right] \left[\frac{\partial}{\partial t} + \frac{i}{2}\left( \frac{\partial^2}{\partial x^2} - \mu \frac{\partial}{\partial x} \right) \right]u\nonumber \\
= & \frac{\partial^2 u}{\partial t^2} + \frac{1}{2^2}\left( \frac{\partial^2}{\partial x^2} - \mu \frac{\partial}{\partial x}  \right)^2 u. \label{Srdmu}
\end{align}
The fundamental solution to \eqref{Srdmu} takes the form
\begin{equation}
u(x,t) = \frac{e^{\mu x}}{\sqrt{2\pi t}}  \cos \left( \frac{x^2}{2t} - \frac{\mu^2 t}{2} - \frac{\pi}{4} \right)
\end{equation}
which slightly extends \eqref{urod2}. 
\end{os}

\section{Pseudo-Processes related to the equation of vibrations of rods}
We introduce a pseudo-process associated to the fundamental solution of equation \eqref{eq1} as follows. We can assume that at time $t=0$ we choose randomly the direction of time (increasing with probability $1/2$) measuring with imaginary values. On each branch of the time axis we consider a Brownian motion, that is a process with independent increments over non-overlapping time intervals. To the vector $\bigg( F(t_1), \ldots , F(t_n) \bigg)$ with $0 <  t_1 < \ldots<t_j < \ldots < t_n$ we associate the signed measure
\begin{align}
& \mu\{ F(t_1) \in dx_1, \ldots , F(t_n) \in dx_n \} \label{frenmultiv}\\
= & \frac{1}{2} \bigg[ Pr\{B(s_1) \in dx_1, \ldots , B(s_n) \in dx_n\} \Big|_{s_j = it_j} + Pr\{B(s_1) \in dx_1, \ldots , B(s_n) \in dx_n \}\Big|_{s_j=-it_j} \bigg]\nonumber \\
= & \frac{1}{2} \left[ \prod_{j=1}^n \frac{e^{-\frac{(x_j-x_{j-1})^2}{2i(t_j- t_{j-1})}}}{\sqrt{2\pi i (t_j-t_{j-1})}} dx_j + \prod_{j=1}^n \frac{e^{-\frac{(x_j-x_{j-1})^2}{-2i(t_j-t_{j-1})}}}{\sqrt{-2\pi i (t_j-t_{j-1})}}dx_j \right]  \nonumber \\
=& \frac{(2\pi)^{-n/2}}{2 \prod_{j=1}^n \sqrt{t_j - t_{j-1}} } \left[ e^{-i \frac{\pi}{4}n} e^{- \sum_{j=1}^n \frac{(x_j-x_{j-1})^2}{2i(t_j- t_{j-1})}}+ e^{i\frac{\pi}{4}n} e^{-\frac{(x_j-x_{j-1})^2}{-2i(t_j-t_{j-1})}} \right] \prod_{j=1}^n dx_j  \nonumber \\
=& \frac{(2\pi)^{-n/2}}{\prod_{j=1}^n \sqrt{t_j - t_{j-1}} }\cos\left( \sum_{j=1}^n \frac{(x_j - x_{j-1})^2}{2 (t_j-t_{j-1})} - n\frac{\pi}{4} \right)  \prod_{j=1}^n dx_j .\nonumber
\end{align}
We can verify that
\[ \int_{\mathbb{R}} \ldots  \int_{\mathbb{R}} \mu\{ F(t_1) \in dx_1, \ldots , F(t_n) \in dx_n \} =1 \]
by writing \eqref{frenmultiv} as
\begin{align}
& \frac{(2\pi)^{-n/2}}{\prod_{j=1}^n \sqrt{t_j - t_{j-1}} } \Bigg[ \cos\left( \sum_{j=1}^{n-1} \frac{(x_j - x_{j-1})^2}{2 (t_j-t_{j-1})} - (n-1) \frac{\pi}{4} \right) \cos\left( \frac{(x_n - x_{n-1})^2}{2(t_n - t_{n-1})} - \frac{\pi}{4}\right) \nonumber \\
& - \sin\left( \sum_{j=1}^{n-1} \frac{(x_j - x_{j-1})^2}{2 (t_j-t_{j-1})} - (n-1) \frac{\pi}{4} \right) \sin\left( \frac{(x_n - x_{n-1})^2}{2(t_n - t_{n-1})} - \frac{\pi}{4}\right)  \Bigg] \prod_{j=1}^n dx_j .\label{reqLam}
\end{align}
By integrating \eqref{reqLam} with respect to $x_n$ the second term vanishes while the first one by applying the Fresnel integrals reduces to
\begin{equation}
\frac{(2\pi)^{-(n-1)/2}}{\prod_{j=1}^{n-1} \sqrt{t_j - t_{j-1}} } \cos\left( \sum_{j=1}^{n-1} \frac{(x_j - x_{j-1})^2}{2 (t_j-t_{j-1})} - (n-1) \frac{\pi}{4} \right) \prod_{j=1}^{n-1} dx_j .\label{densMar}
\end{equation}
The density \eqref{densMar} is the marginal $\mu\{ \cap_{j=1}^{n-1} F(t_j) \in dx_j \}$ of the vector 
\[ \left( F(t_1), \ldots , F(t_{n-1}) \right).\] 
The one- and two-dimensional densities of the Fresnel pseudo-process read
\begin{equation}
\mu\{ F(t_1) \in dx_1 \} = \frac{dx_1}{\sqrt{2\pi t_1}} \cos\left( \frac{x^2_1}{2t_1} - \frac{\pi}{4} \right)
\end{equation}
and
\begin{equation}
\mu\{ F(t_1) \in dx_1 , F(t_2) \in dx_2\} = \frac{dx_1dx_2}{2\pi \sqrt{t_1(t_2 - t_1)}} \sin\left( \frac{x_1^2}{2t_1} + \frac{(x_2 -x_1)^2}{2(t_2 -t_1)} \right).
\end{equation}
By means of \eqref{frenmultiv} we can construct the signed measure of cylinder sets of the form
\[ C = \left\lbrace \bigcap_{j=1}^n \left( a_j \leq x_j \leq b_j\right) \right\rbrace  \]
in the following manner
\begin{align}
\mu\{C\} = & \int_{a_1}^{b_1} \ldots \int_{a_n}^{b_n} \mu\{F(t_1) \in dx_1, \ldots, F(t_n) \in dx_n \}\label{cylinder}\\
= & \int_{a_1}^{b_1} \ldots \int_{a_n}^{b_n} \frac{(2\pi)^{-n/2}}{\prod_{j=1}^n  \sqrt{t_j - t_{j-1}}} \cos\left( \sum_{j=1}^n \frac{(x_j -x_{j-1})^2}{2(t_j - t_{j-1})} - n \frac{\pi}{4} \right) \prod_{j=1}^n dx_j .\nonumber
\end{align}
The construction of \eqref{cylinder} follows the same line of the signed measures of pseudo-processes related to higher-order heat equations (see for example, Krylov \cite{Kry60}, Ladokhin \cite{Lad63,Lad64}, Daletsky et al. \cite{Dali64}, Hochberg \cite{Hoc78}). The signed measure is extended to the $\sigma$-algebra generated by the cylinder sets in the usual way (see, for example \cite{Lad63}).

\begin{os}
\normalfont
The signed measure constructed above is not markovian since
\begin{equation} 
\begin{array}{l} \mu\{ F(t_1) \in dx_1,  F(t_3) \in dx_3 \big| F(t_2) = x_2 \} \\
\neq \mu\{ F(t_1) \in dx_1 \big| F(t_2) = x_2 \} \mu \{F(t_3) \in dx_3 \big| F(t_2) = x_2\}
\end{array} \label{statem} \end{equation}
with $t_1 < t_2 < t_3$. From \eqref{frenmultiv} we have that
\begin{align*} 
\mu\{ F(t_1) \in dx_1,  F(t_3) \in dx_3 \Big| F(t_2) = x_2 \} = & \frac{\mu\{ F(t_1) \in dx_1,  F(t_2) = x_2, F(t_3) \in dx_3 \}}{\mu\{ F(t_3) \in dx_3 \}}\\
= & \frac{\sqrt{t_2}}{2\pi} \frac{\cos\left( \frac{(x_1 -x_0)^2}{2|t_1 - t_0|} + \frac{(x_2 - x_1)^2}{2|t_2 - t_1|} + \frac{(x_3 - x_2)^2}{2|t_3 - t_2|} - 3 \frac{\pi}{4} \right)}{\sqrt{\prod_{j=1}^3 (t_j - t_{j-1})} \, \cos\left(\frac{x_2^2}{2t_2} - \frac{\pi}{4}  \right)}\, dx_1\, dx_3.
\end{align*}
Clearly 
\begin{align*}
& \mu\{ F(t_1) \in dx_1 \Big| F(t_2) = x_2 \} \mu \{F(t_3) \in dx_3 \Big| F(t_2) = x_2\}\\
= & \sqrt{t_2} \, \frac{\cos\left( \frac{x_1^2}{2t_1}+ \frac{(x_2 - x_1)^2}{2|t_2 -t_1|} - \frac{\pi}{2} \right) \cos\left( \frac{x_2^2}{2t_2} + \frac{(x_3 - x_2)^2}{2|t_3 - t_2|} - \frac{\pi}{2} \right)}{\sqrt{t_1 (t_3- t_2)(t_2-t_1)} \, \cos^2\left( \frac{x_2^2}{2t_2} - \frac{\pi}{4} \right)}\, dx_1\, dx_3
\end{align*}
and thus we conclude that the statement \eqref{statem} holds.
\end{os}

\subsection{Feynman-Kac formula}
For a non-negative $k \in C^2(\mathbb{R})$, the Feynman-Kac functional
\begin{equation}
w(x,t) = E \left[ e^{-\int_0^t k(F(s)) ds} \big| F(0)=x \right] 
\end{equation}
must be understood in the sense that
\begin{equation}
\lim_{n \to \infty} \int_{\mathbb{R}} \ldots \int_{\mathbb{R}} \exp\left( - \sum_{j=1}^{n} k(x_j)(t_j - t_{j-1}) \right)  \mu\left\lbrace F(t_1) \in dx_1, \ldots, F(t_n) \in dx_n \right\rbrace  \label{measkac}
\end{equation}
where $\mu$ is the signed measure defined in \eqref{frenmultiv} and provided that the limit exists. This is similar to the definition of Feynman-Kac functional in Krylov \cite{Kry60}, Lachal \cite{Lach03}, Hochberg \cite{Hoc78}. We now show that the functional $w$ solves the p.d.e.  
\begin{equation}
\frac{\partial^2 w}{\partial t^2}(x,t) = -\frac{1}{2} \left[ \frac{\partial^4 w}{\partial x^4}(x, t) - \frac{\partial^2}{\partial x^2}\Big(k(x)w(x,t) \Big) - k(x) \frac{\partial^2 w}{\partial x^2}(x,t) - k^2(x) w(x,t) \right]. \label{kacf}
\end{equation}
In order to prove \eqref{kacf} we consider that the measure $\mu$ appearing in \eqref{measkac} permits us to write $w(x,t)$ as
\[ w(x,t) = \frac{1}{2}\left[ \hat{w}(x, it) + \hat{w}(x, -it) \right] \]
where
\begin{equation}
\hat{w}(x,t) = E\left[ e^{-\int_0^t k(B(s)) ds} \Big| B(0) = x \right]
\end{equation}
is the classical Feynman-Kac function for Brownian motion. Since 
\[ \frac{\partial \hat{w}}{\partial t}(x,t) = \frac{1}{2} \frac{\partial^2 \hat{w}}{\partial x^2}(x,t) - k(x) \hat{w}(x,t) \]
by deriving \eqref{kacf} with respect to $t$ we have that
\begin{align}
\frac{\partial w}{\partial t}(x,t) = & \frac{i}{2} \frac{\partial^2 \hat{w}}{\partial x^2}(x, it) - i k(x) \hat{w}(x, it) - \frac{i}{2} \frac{\partial^2 \hat{w}}{\partial x^2}(x, -it) + i k(x) \hat{w}(x, -it). \label{dfql}
\end{align}
One more time-derivative in \eqref{dfql} yields
\begin{align*}
\frac{\partial^2 w}{\partial t^2}(x,t) = & \frac{1}{2}\Bigg[ \frac{i}{2} \Big[ i\frac{\partial^2}{\partial x^2}\left( \frac{1}{2} \frac{\partial^2 \hat{w}}{\partial x^2}(x,it) - k(x) \hat{w}(x,it) \right) - ik(x) \left( \frac{1}{2} \frac{\partial^2 \hat{w}}{\partial x^2}(x,it) - k(x) \hat{w}(x,it)\right)\\
& + i \frac{\partial^2}{\partial x^2}\left( \frac{1}{2} \frac{\partial^2 \hat{w}}{\partial x^2}(x,-it) - k(x) \hat{w}(x,-it) \right) -i k(x) \left( \frac{1}{2} \frac{\partial^2 \hat{w}}{\partial x^2}(x,-it) - k(x) \hat{w}(x,-it) \right) \Big] \Bigg].
\end{align*}
We observe that for equation \eqref{kacf} a decoupling similar to that applied before in Section 3 and 4 does not work because of the form of the first-order time derivative \eqref{dfql}.

\subsection{Superposition of vibrations}
We now analyse the superposition of solutions of the equation \eqref{eq1}. Our approach is based on Fourier transforms and gives
\begin{align*}
& \int_{-\infty}^{+\infty} \ldots \int_{-\infty}^{+\infty} e^{i \beta \sum_{j=1}^{n} x_j} \prod_{j=1}^{n} \frac{1}{\sqrt{2\pi t}} \cos\left( \frac{x^2_j}{2t} - \frac{\pi}{4} \right) dx_j\\
= & \Bigg[ \int_{-\infty}^{+\infty} e^{i\beta x} \frac{1}{\sqrt{2\pi t}} \cos\left( \frac{x^2}{2t} - \frac{\pi}{4} \right) dx  \Bigg]^n = \cos^n \frac{\beta^2 t}{2} = \frac{\left(e^{i \frac{\beta^2 t}{2} } + e^{-i \frac{\beta^2 t}{2}}\right)^n}{2^n}\\
= & \frac{1}{2^n} \sum_{k=0}^n \binom{n}{k} e^{-i \frac{\beta^2 t}{2}(n - 2k)}\\
= & \frac{1}{2^{n-1}} \left[ \frac{e^{i \frac{\beta^2 t}{2}n} + e^{-i \frac{\beta^2 t}{2} n}}{2} + \frac{1}{2} \sum_{k=1}^{n-1} \binom{n}{k}  e^{-i \frac{\beta^2 t}{2}(n - 2k)} \right] \\
= & \frac{1}{2^{n-1}} \left[ \frac{e^{i \frac{\beta^2 t}{2}n} + e^{-i \frac{\beta^2 t}{2} n}}{2} + \binom{n}{1} \frac{e^{i \frac{\beta^2 t}{2}(n - 2)} + e^{-i \frac{\beta^2 t}{2} (n - 2)}}{2} + \frac{1}{2}\sum_{k=2}^{n-2} \binom{n}{k} e^{-i \frac{\beta^2 t}{2}(n - 2k)} \right].
\end{align*}
For $n \in 2 \mathbb{N}$ we obtain that
\begin{align}
\cos^n \frac{\beta^2 t}{2} = & \frac{1}{2^{n-1}} \sum_{k=0}^{n/2-1} \binom{n}{k} \cos \frac{\beta^2 t (n - 2k)}{2} + \frac{1}{2^n} \binom{n}{n/2} \label{pariSum}
\end{align}
whereas, for $n \in 2\mathbb{N} +1$ we have that
\begin{equation}
\cos^n \frac{\beta^2 t}{2} = \frac{1}{2^{n-1}} \sum_{k=0}^{(n-1)/2} \binom{n}{k} \cos \frac{\beta^2 t (n - 2k)}{2}.
\end{equation}
Formula \eqref{pariSum} shows that for $n$ even the resulting superpositions of waves consists of components of the form \eqref{funQW} plus a Dirac delta function.

\begin{os}
\normalfont
For an even number of terms in the sum $\sum_{j=1}^n F_j(t)$ we have delta components in zero. This surprising fact can be also confirmed by considering the convolution of two Fresnel waves as shown below.
\begin{align*}
& \int_{-\infty}^{+\infty} u(w, t) u(x-w, t) dw \\
= & \int_{-\infty}^{+\infty} \frac{1}{\sqrt{2\pi t}} \cos\left( \frac{w^2}{2t} - \frac{\pi}{4} \right) \frac{1}{\sqrt{2\pi t}} \cos\left( \frac{(x-w)^2}{2t} - \frac{\pi}{4} \right) dw \\
= & \frac{1}{2^2 \pi t} \Bigg[ \int_{-\infty}^{+\infty} \cos\left( \frac{(x-w)^2}{2t} - \frac{w^2}{2t} \right) + \sin\left( \frac{(x-w)^2}{2t} + \frac{w^2}{2t} \right) dw \Bigg]\\
= & \frac{1}{2^2 \pi t} \Bigg[ \int_{-\infty}^{+\infty}\cos\left( \frac{x^2}{2t} - \frac{wx}{t} \right) + \sin\left( \frac{x^2}{2t} + \frac{w^2}{t} - \frac{wx}{t} \right) dw \Bigg]\\
= & \frac{1}{4\pi t} \Bigg[ \cos \frac{x^2}{2t} \int_{-\infty}^{+\infty} \cos \frac{xw}{t} dw + \sin \left( \frac{(w-x)^2}{t} - \frac{x^2}{t} + \frac{wx}{t} \right) dw \Bigg]\\
= & \frac{1}{2} \delta(x) \cos \frac{x^2}{2t} + \frac{1}{4\pi t} \int_{-\infty}^{+\infty} \sin\left( \frac{w^2}{t} - \frac{x^2}{2t} + \frac{x(w+x)}{t} \right) dw\\
= & \frac{1}{2}\delta(x) + \frac{1}{4\pi t} \int_{-\infty}^{+\infty} \sin\left( \frac{x^2}{2t} + \left( \frac{w}{\sqrt{t}} + \frac{x}{\sqrt{t}} \right)^2 - \frac{x^2}{4t} \right) dw\\
=& \frac{1}{2}\delta(x) + \frac{1}{4\pi t} \int_{-\infty}^{+\infty} \Big[ \sin \frac{x^2}{4t}\cos\left( \frac{w + x/2}{\sqrt{t}} \right)^2 + \cos \frac{x^2}{4t} \sin\left( \frac{w +x/2}{\sqrt{t}} \right)^2 \Big] dw\\
=& \frac{1}{2}\delta(x) + \frac{1}{4\pi \sqrt{t}}  \Bigg[ \sin \frac{x^2}{4t} \int_{-\infty}^{+\infty} \cos\left( w^2 \right) dw +  \cos \frac{x^2}{4t} \int_{-\infty}^{+\infty} \sin\left( w^2 \right)dw \Bigg] \\
=& \frac{1}{2}\delta(x) + \frac{1}{4 \sqrt{\pi t}}  \Bigg[ \frac{1}{\sqrt{2}} \sin \frac{x^2}{4t} +  \frac{1}{\sqrt{2}}\cos \frac{x^2}{4t} \Bigg] \\
= & \frac{1}{2}\delta(x) + \frac{1}{2} \frac{1}{\sqrt{4\pi t}} \cos\left( \frac{x^2}{4t} - \frac{\pi}{4} \right).
\end{align*}
This result accords with \eqref{pariSum} for $n=2$ after an inversion of the Fourier transform.
\end{os}

\section{Vibrations of plates}
\subsection{Vibrations of infinite plates}
The $d$-dimensional version of equation \eqref{statem} has the form
\begin{equation}
\frac{\partial^2 u}{\partial t^2} = -\frac{1}{2^2}\left( \frac{\partial^2}{\partial x_1^2} + \ldots + \frac{\partial^2}{\partial x_d^2}\right)^2 u \label{pdeMULT}
\end{equation}
and emerges in the study of vibrations of rigid thin structures like plates. The solution to \eqref{pdeMULT} subject to the initial conditions
\begin{equation}
\left\lbrace \begin{array}{l} u(x_1, \ldots , x_d, 0) = \prod_{j=1}^d \delta(x_j)\\ u_{t}(x_1, \ldots , x_d, 0)=0 \end{array} \right .
\end{equation}
has the form
\begin{equation}
u(x_1, \ldots , x_d, t) = \frac{1}{(\sqrt{2\pi t})^d} \cos \left( \sum_{j=1}^d \frac{x^2_j}{2t} - d\frac{\pi}{4} \right).  \label{ddim}
\end{equation}
The equation of vibrating plates is examined in the book by Courant and Hilbert \cite[page 307]{CouHilb}  where the case of circular plates with Neumann boundary condition is outlined. 
\begin{te}
The Fourier transform of \eqref{ddim} is
\begin{align}
U(\beta_1, \ldots , \beta_n, t) = & \int_{\mathbb{R}^d} e^{i \sum_{j=1}^d \beta_j\, x_j} u(x_1, \ldots , x_d, t) \, dx_1\ldots dx_d = \cos\left( \sum_{j=1}^d \frac{\beta^2_j\, t}{2} \right) \label{Fourdim}
\end{align}
\end{te}
\begin{proof}
Let us write \eqref{Fourdim} as
\begin{align*}
& \frac{1}{(\sqrt{2\pi t})^d}\Bigg[ \int_{\mathbb{R}^{d-1}} e^{i \sum_{j=1}^{d-1} \beta_j x_j} \cos\left( \frac{\sum_{j=1}^{d-1} x_j^2}{2t} - (d-1)\frac{\pi}{4} \right) \prod_{j=1}^{d-1} dx_j \, \int_{-\infty}^{+\infty} e^{i\beta_d x_d} \cos\left( \frac{x^2_d}{2t} - \frac{\pi}{4} \right) \, dx_d\\
- & \int_{\mathbb{R}^{d-1}} e^{i \sum_{j=1}^{d-1}\beta_{j} x_j } \sin\left( \frac{\sum_{j=1}^{d-1} x_j^2}{2t} - (d-1)\frac{\pi}{4} \right) \prod_{j=1}^{d-1} dx_j \, \int_{-\infty}^{+\infty} e^{i\beta_d x_d} \sin\left( \frac{x^2_d}{2t} - \frac{\pi}{4} \right) \, dx_d \Bigg]\\
= & \frac{1}{(\sqrt{2\pi t})^{d-1}} \int_{\mathbb{R}^{d-1}} e^{i \sum_{j=1}^{d-1} \beta_j x_j} \cos\left(\frac{\sum_{j=1}^{d-1} x_j^2}{2t} - (d-1)\frac{\pi}{4} + \frac{\beta^2_d t}{2} \right) \prod_{j=1}^{d-1} dx_j
\end{align*}
because 
\begin{align}
\frac{1}{\sqrt{2\pi t}} \int_{-\infty}^{+\infty} e^{i \beta_d x_d} \cos\left( \frac{x_d^2}{2t} - \frac{\pi}{4} \right) dx_d = & \cos \frac{\beta_d^2 t}{2},\\
\frac{1}{\sqrt{2\pi t}} \int_{-\infty}^{+\infty} e^{i \beta_d x_d} \sin\left( \frac{x_d^2}{2t} - \frac{\pi}{4} \right) dx_d = & \sin \frac{\beta_d^2 t}{2}.
\end{align}
After $m$ integrations, at the $(m+1)$-th step, we have that
\begin{align*}
& \frac{1}{(\sqrt{2\pi t})^{d-m}}\Bigg[ \int_{\mathbb{R}^{d - m-1}} e^{i \sum_{j=1}^{d-m-1} \beta_j x_j} \cos\left( \frac{\sum_{j=1}^{d -m-1} x_j^2}{2t}  - (d-m-1) \frac{\pi}{4} + \sum_{j=d-m+1}^{d} \frac{\beta_j^2 t}{2} \right) \prod_{j=1}^{d-m-1} dx_j\\
& \int_{-\infty}^{+\infty} e^{i\beta_{d-m} x_{d-m}} \cos\left( \frac{x^2_{d-m}}{2t} - \frac{\pi}{4} \right) dx_{d-m}\\
& - \int_{\mathbb{R}^{d-m-1}} e^{i \sum_{j=1}^{d-m-1}\beta_j x_j} \sin\left( \frac{\sum_{j=1}^{d-m-1} x_j^2}{2t}  - (d-m-1) \frac{\pi}{4} + \sum_{j=d-m+1}^{d} \frac{\beta_j^2 t}{2} \right) \prod_{j=1}^{d-m-1} dx_j\\
& \int_{-\infty}^{+\infty} e^{i\beta_{d-m} x_{d-m}} \sin\left( \frac{x^2_{d-m}}{2t} - \frac{\pi}{4} \right) dx_{d-m} \Bigg]\\
= & \frac{1}{(\sqrt{2\pi t})^{d-m-1}} \int_{\mathbb{R}^{d - m - 1}} e^{i \sum_{j=1}^{d-m-1} \beta_j x_j} \cos\left( \frac{\sum_{j=1}^{d -m -1} x_j^2}{2t}  - (d-m-1) \frac{\pi}{4} + \sum_{j=d-m}^{d} \frac{\beta_j^2 t}{2} \right) \prod_{j=1}^{d-m-1} dx_j .
\end{align*}
By performing the $d-m-1$ remaining integrals with respect to $x_1, \ldots , x_{d-m-1}$ we arrive at result \eqref{Fourdim}.
\end{proof}
In particular, for $\beta_1= \ldots =\beta_n = \beta$, from \eqref{Fourdim} we have that
\begin{equation}
U(\beta, \ldots , \beta, t) = \int_{\mathbb{R}^d} e^{i \beta \sum_{j=1}^d  x_j} u(x_1, \ldots , x_d, t) \, dx_1\ldots dx_d = \cos \frac{\beta^2 d\, t}{2}
\end{equation}
which shows that the sum of the marginals of \eqref{ddim} has the same form of \eqref{funQW} at time $t d$. For $d=1$ the equation \eqref{pdeMULT} coincides with \eqref{eq1}, \eqref{ddim} becomes \eqref{funQW} and \eqref{Fourdim} reduces to \eqref{jkl}. For $d=2$ we obtain the interesting result
\begin{equation}
u(x_1, x_2, t) = \frac{1}{2\pi t} \sin \left( \frac{x_1^2 + x_2^2}{2t} \right) \label{law2dim}
\end{equation} 
and the corresponding structure is depicted in figure \ref{sdFig2}.

\begin{figure}[ht]
\centering
\includegraphics[width=12cm, height=5cm]{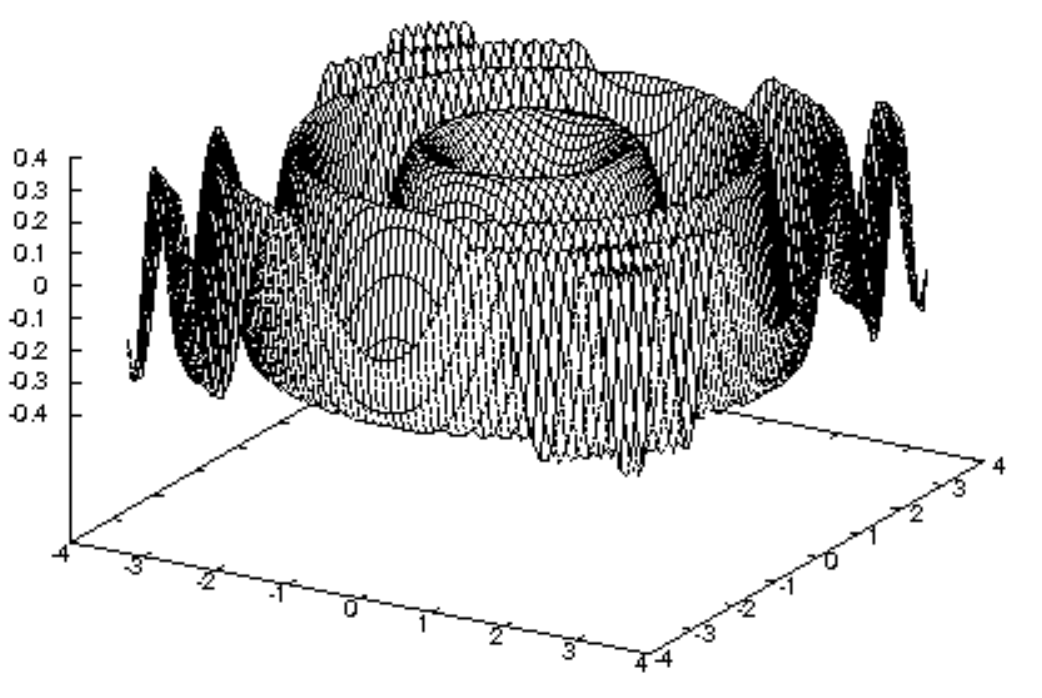} 
\caption{The vibrating surface \eqref{law2dim} (within a sphere) at time $t=0$}
\label{sdFig2}
\end{figure}

It is apparent from \eqref{law2dim} that the bivariate signed measure does not factorize. We can check that \eqref{law2dim} (as well as \eqref{ddim}) integrates to unity
\begin{align*}
\int_{-\infty}^{+\infty} \int_{-\infty}^{+\infty} u(x_1,x_2,t) \, dx_1\, dx_2 = & \frac{1}{2\pi t} \int_{-\infty}^{+\infty} \int_{-\infty}^{+\infty} \sin\left( \frac{x_1^2 + x_2^2}{2t} \right)\, dx_1\, dx_2\\
= & \frac{1}{2\pi t} \int_{\mathbb{R}^2} \left[ \sin \frac{x_1^2}{2t}\cos \frac{x_2^2}{2t} + \sin \frac{x_2^2}{2t}\cos \frac{x_1^2}{2t} \right] dx_1\, dx_2\\
= & \frac{1}{\pi t} \left( \int_{\mathbb{R}}\sin \frac{x^2}{2t} \, dx \right)^2 = 1.
\end{align*}
It is worthwhile to note also that
\[ u(x_1, t)\, u(x_2, t)  \neq u(x_1, x_2, t) \]
because
\begin{align*}
u(x_1, t)\, u(x_2, t) = & \frac{1}{2\pi t} \cos\left( \frac{x_1^2}{2t} - \frac{\pi}{4} \right)\, \cos\left( \frac{x_2^2}{2t} - \frac{\pi}{4} \right)\\
= & \frac{1}{2^2 \pi t} \left[ \sin\left( \frac{x_1^2 +x_2^2}{2t} \right) + \cos\left( \frac{x_1^2 - x_2^2}{2t} \right) \right]\\
= & \frac{1}{2} u(x_1, x_2, t) + \frac{1}{2^2 \pi t} \cos\left( \frac{x_1^2 - x_2^2}{2t} \right).
\end{align*}

\subsection{Vibrations of circular plates}
For the analysis of the vibrations of circular plates we need the following results from the theory of reflecting Brownian motion inside a circle $C_R$. Since the problem has isotropic structure (the circular invariance is due to the initial disturbance concentrated in the center starting off the vibrations) we can restrict ourselves to the Cauchy problem for the heat equation
\begin{equation}
\frac{\partial u}{\partial t} = \frac{1}{2}\left[ \frac{\partial^2}{\partial r^2} + \frac{1}{2}\frac{\partial}{\partial r} \right] u, \quad 0< r < R,\; t>0
\label{eqasd}
\end{equation}
with boundary and initial conditions
\begin{equation}
\left\lbrace \begin{array}{ll} u(r,0) = \delta(r) \\ \frac{\partial u}{\partial n} \Big|_{\partial C_{R}} = 0. \end{array} \right .
\label{eqasdcond}
\end{equation}
The solution of \eqref{eqasd} with \eqref{eqasdcond} is 
\begin{equation}
q^{ref}(r,t) = \frac{1}{t}\left[ e^{-\frac{r^2}{2t}} + e^{-\frac{R^4}{2r^2 t}} \right]
\end{equation}
and is constructed by means of the inversion of radius. The probability law of the reflecting Brownian motion in $C_R$ solves the adjoint equation
\begin{equation}
\frac{\partial u}{\partial t} = \frac{1}{2} \left[ \frac{\partial^2}{\partial r^2} - \frac{1}{r} \frac{\partial}{\partial r} + \frac{1}{r^2} \right]u \label{seiundici}
\end{equation}
with initial and boundary conditions \eqref{eqasdcond} and has the following explicit form
\begin{equation}
p^{ref}(r,t) = \frac{r}{t} e^{-\frac{r^2}{2t}} + \frac{R^4}{r^3 t} e^{-\frac{R^4}{2r^2 t}}, \quad R>r>0,\; t>0
\end{equation}
and can be checked that it integrates to unity. The results above permit us to solve related boundary value problems for the equation of vibrations of plates. We have the following result.
\begin{te}
The Cauchy problem
\begin{equation}
\frac{\partial^2 u}{\partial t^2} = -\frac{1}{2^2}\left[ \frac{\partial^2}{\partial r^2} + \frac{1}{r}\frac{\partial}{\partial r} \right]^2 u, \quad 0< r < R,\; t>0
\end{equation}
with initial and boundary conditions
\begin{equation}
\left\lbrace \begin{array}{ll} u(r,0) = \delta(r) \\ u_t(x,0)=0\\ \frac{\partial u}{\partial n} \Big|_{\partial C_{R}} = 0 \end{array} \right .
\label{condwww}
\end{equation}
has solution
\begin{equation}
\bar{q}^{ref}(r,t) = \frac{1}{t} \sin \frac{r^2}{2t} + \frac{1}{t}\sin \frac{R^4}{2r^2 t} \label{sequattordici}
\end{equation}
while the adjoint equation
\begin{equation}
\frac{\partial^2 u}{\partial t^2} = -\frac{1}{2^2}\left[ \frac{\partial^2}{\partial r^2} -  \frac{1}{r}\frac{\partial}{\partial r} + \frac{1}{r^2} \right]^2 u, \quad 0< r < R,\; t>0 \label{seiquindici}
\end{equation}
with initial and boundary conditions \eqref{condwww} has solution
\begin{equation}
\bar{p}^{ref}(r,t) = \frac{r}{t} \sin \frac{r^2}{2t} + \frac{R^4}{r^3 t} \sin \frac{R^4}{2r^2 t}.
\label{seisedici}
\end{equation}
\end{te}
\begin{proof}
We give the proof of \eqref{seisedici} since the same method easily leads to \eqref{sequattordici}. The equation appearing in \eqref{seiquindici} can be decoupled as
\begin{equation}
\left[ \frac{\partial}{\partial t} + \frac{i}{2}\left( \frac{\partial^2}{\partial r^2} -  \frac{1}{r}\frac{\partial}{\partial r} + \frac{1}{r^2} \right) \right] \left[ \frac{\partial}{\partial t} - \frac{i}{2}\left( \frac{\partial^2}{\partial r^2} -  \frac{1}{r}\frac{\partial}{\partial r} + \frac{1}{r^2} \right) \right] u = 0
\end{equation}
and each Schr\"odinger-type equation is formally similar to \eqref{seiundici} and we can thus write that
\begin{align*}
\bar{p}^{ref}(r,t) = & \frac{1}{2} \left[ p^{ref}(r, it) + p^{ref}(r, -it) \right]\\
= & \frac{1}{2}\left[ \frac{r}{it} e^{- \frac{r^2}{2it}} + \frac{R^4}{r^3 it} e^{- \frac{R^4}{2r^2 it}} + \frac{r}{-it} e^{-\frac{r^2}{-2it}} + \frac{R^4}{-2 r^3 it} e^{-\frac{R^4}{-2r^2 it}} \right]\\
= & \frac{r}{t}\sin \frac{r^2}{2t} + \frac{R^4}{r^3 t} \sin \frac{R^4}{2r^2 t}.
\end{align*}
\end{proof}
The isotropic structure of the solution permits us to write the form of the circular vibrating plate as
\begin{equation}
p(r,\theta, t) = \frac{1}{2\pi} \bar{p}^{ref}(r,t) = \frac{1}{2\pi t} \left[ r \sin \frac{r^2}{2t} + \frac{R^4}{r^3} \sin \frac{R^4}{2r^2 t} \right], \quad 0<r<R,\; 0< \theta < 2\pi. \label{expression618}
\end{equation}
By considering the change of variable $r^\prime = R^2/r$ in the following integral we get
\begin{align*}
\int_{C_R} p(r, \theta, t) \, dr\,d\theta = &  \frac{1}{2\pi t} \int_0^{2\pi} \int_0^R \left[ \frac{r}{t} \sin \frac{r^2}{2t} + \frac{R^4}{r^3 t} \sin \frac{R^4}{2 r^2 t} \right] dr\, d\theta\\
= & \frac{1}{2\pi t} \int_0^{2\pi}\left[ \int_0^R r \sin \frac{r^2}{2t} dr + \int_R^\infty r \sin \frac{r^2}{2t} dr \right] d\theta\\
= & \frac{1}{2\pi t} \int_0^{2\pi} \int_0^\infty r \sin \frac{r^2}{2t} dr d\theta =  \frac{1}{2\pi t} \int_{\mathbb{R}^2} \sin\frac{x^2 +y^2}{2t} dxdy\\
= & \frac{2(\sqrt{2t})^2}{2 \pi t} \left( \int_{-\infty}^{+\infty} \cos y^2 \, dy \right)^2 = 1.
\end{align*}

\begin{os}
\normalfont
The expression \eqref{expression618} in cartesian coordinates reads
\begin{equation}
p(x,y,t) = \frac{1}{2\pi t} \left[ \sin \frac{x^2 + y^2}{2t} + \frac{R^4}{(x^2+y^2)^4} \sin \frac{R^4}{2t (x^2 + y^2)^2} \right], \quad (x,y) \in C_R,\; t>0
\end{equation}
while the expression of $q(r, \theta, t) = \frac{1}{2\pi} \bar{q}^{ref}(r,t)$ in cartesian coordinates reads
\begin{equation}
q(x,y,t) = \frac{1}{2\pi t} \frac{1}{\sqrt{x^2 + y^2}} \left[ \sin \frac{x^2 + y^2}{2t} + \sin \frac{R^4}{2t (x^2 + y^2)^2} \right], \quad (x,y) \in C_R,\; t>0.
\end{equation}
\end{os}

\begin{figure}[ht]
\centering
\includegraphics[scale=0.5]{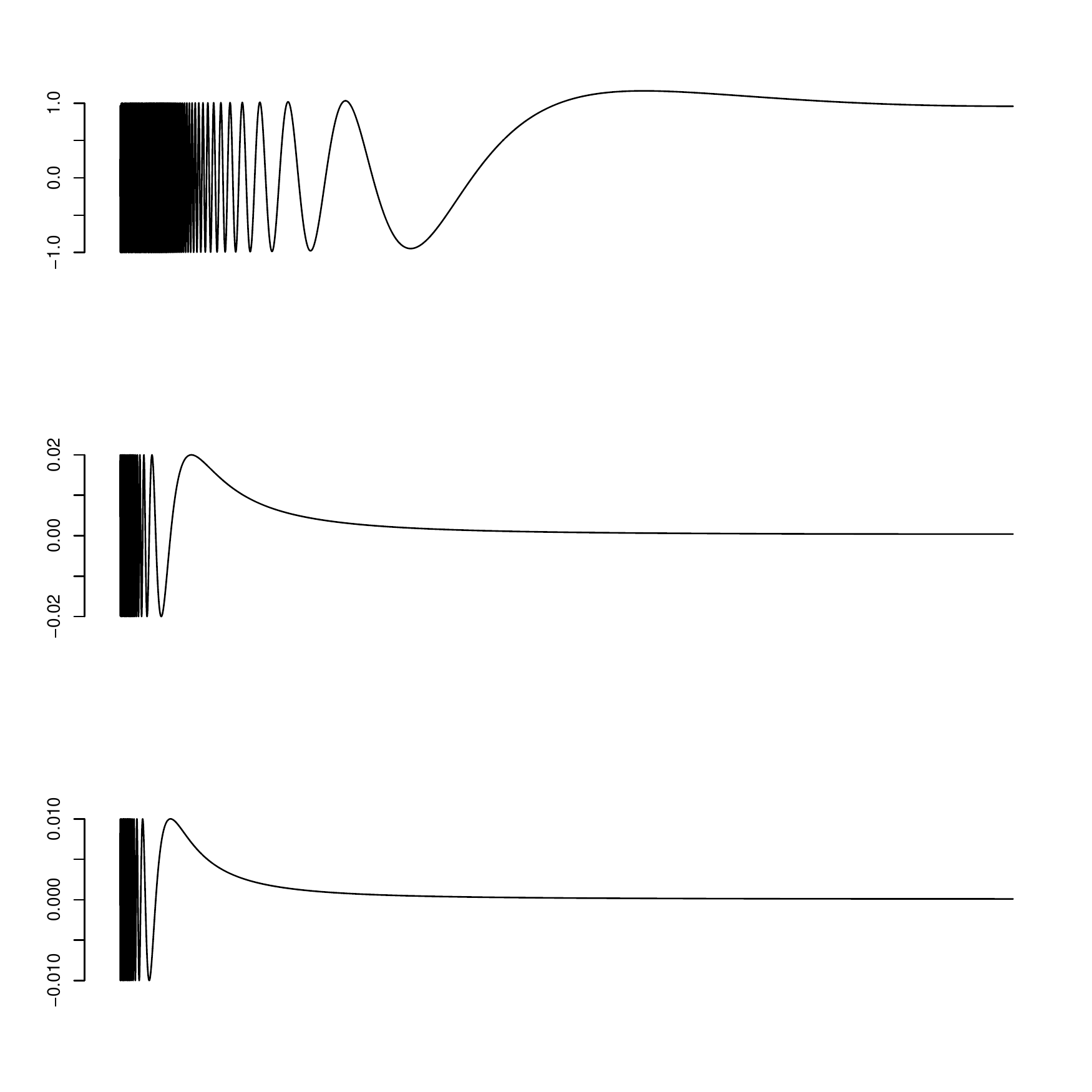}
\caption{The profile of the function \eqref{sequattordici} with $0 < r < R=1$ and $t=1,50,100$.}
\label{fiDg}
\end{figure}

The cross-section of the vibrating disk (formula \eqref{sequattordici}) is depicted in figure \ref{fiDg} in three different instants. It shows that the vibrations initially extend to the whole plate and they fade off outside a neighbourhood of the origin as time passes.

\section{Subordination of the Fresnel pseudo-process with different processes}

The composition of the Fresnel process $F(t)$, $t>0$ with different processes produces interesting results. The first one has been outlined in the previous sections and leads to the following statement concerning pseudo-processes of fourth-order which have been dealt by several authors and from different viewpoints (see  Hochberg \cite{Hoc78}, Nikitin and Orsingher \cite{NO00}, Nishioka \cite{Nis97}, Lachal \cite{Lach03,Lach07}, Benachour et al. \cite{BRV96}).
\begin{te}
The pseudo-process $F(|B(t)|)$, $t>0$ has measure density which satisfies the fourth-order heat equation
\begin{equation}
\left\lbrace \begin{array}{ll} \frac{\partial u}{\partial t} = - \frac{1}{2^3} \frac{\partial^4 u}{\partial x^4} &  x \in \mathbb{R},\, t>0\\ u(x,0)= \delta(x). \end{array} \right .
\label{eqsubFren}
\end{equation}
\end{te}
\begin{proof}
We write
\begin{equation}
q(x,t) =2 \int_0^\infty \frac{1}{\sqrt{2\pi s}} \cos\left( \frac{x^2}{2s} - \frac{\pi}{4} \right) \frac{e^{-\frac{s^2}{2t}}}{\sqrt{2\pi t}} ds. \label{sdhj}
\end{equation}
The Fourier transform of \eqref{sdhj} reads
\begin{align}
\int_{\mathbb{R}} e^{i\beta x}q(x,t)dx = & \int_{\mathbb{R}} e^{i\beta x} \int_{0}^{\infty} \left[ \frac{e^{-\frac{x^2}{2(is)}}}{\sqrt{2\pi (is)}} + \frac{e^{-\frac{x^2}{2(-is)}}}{\sqrt{2\pi (-is)}} \right] \frac{e^{-\frac{s^2}{2t}}}{\sqrt{2\pi t}} ds\label{seitre}\\
= & \int_0^\infty \left[ e^{-i\frac{\beta^2}{2}s }+ e^{i\frac{\beta^2}{2}s} \right] \frac{e^{-\frac{s^2}{2t}}}{\sqrt{2\pi t}} ds =  \int_{-\infty}^{+\infty} e^{-i\frac{\beta^2}{2}s }\frac{e^{-\frac{s^2}{2t}}}{\sqrt{2\pi t}} ds =  \exp\left( - \frac{\beta^4}{2^3}t \right). \nonumber
\end{align}
Clearly \eqref{seitre} satisfies the Fourier transform of the equation \eqref{eqsubFren}. Result \eqref{seitre} can also be obtained by exploting \eqref{jkl} and thus
\begin{align*}
\int_{\mathbb{R}} e^{i \beta x} q(x, t) dx = & 2 \int_0^\infty \left[ \int_{\mathbb{R}} \frac{e^{i \beta x}}{\sqrt{2\pi s}} \cos \left( \frac{x^2}{2s} - \frac{\pi}{4} \right) dx \right] \frac{e^{-\frac{s^2}{2t}}}{\sqrt{2\pi t}} ds\\
= & 2 \int_0^\infty \cos \frac{\beta^2 s}{2} \frac{e^{-\frac{s^2}{2t}}}{\sqrt{2 \pi t}} ds =  \int_{\mathbb{R}} e^{i \beta x} \frac{e^{-\frac{s^2}{2t}}}{\sqrt{2\pi t}} ds =  \exp\left(-\frac{\beta^4 t}{2^3}\right).
\end{align*}
\end{proof}
If $B(t)$, $t>0$ has volatility equal to $\sigma^2$ then the composition $F(|B(t)|)$, $t>0$, has Fourier transform equal to
\[ \int_{\mathbb{R}} e^{i \beta x} q(x,t) dx = \exp\left( -\frac{\beta^4 \sigma^2 t}{2^3} \right) \]
and for $\sigma^2=2$ leads to the pseudo-process with density \eqref{sdhj}. 

\begin{os}
\normalfont
This result has been obtained in different ways in \cite{BRV96, HO96}. Result \eqref{sdhj} shows that the fundamental solution to the fourth-order equation \eqref{eqsubFren} can be viewed as the profile of a vibrating rod with a Gaussian weight which damps the oscillations, see figure \ref{lab4ord}.
\end{os}

\begin{figure}[ht]
\centering
\includegraphics[scale=.5]{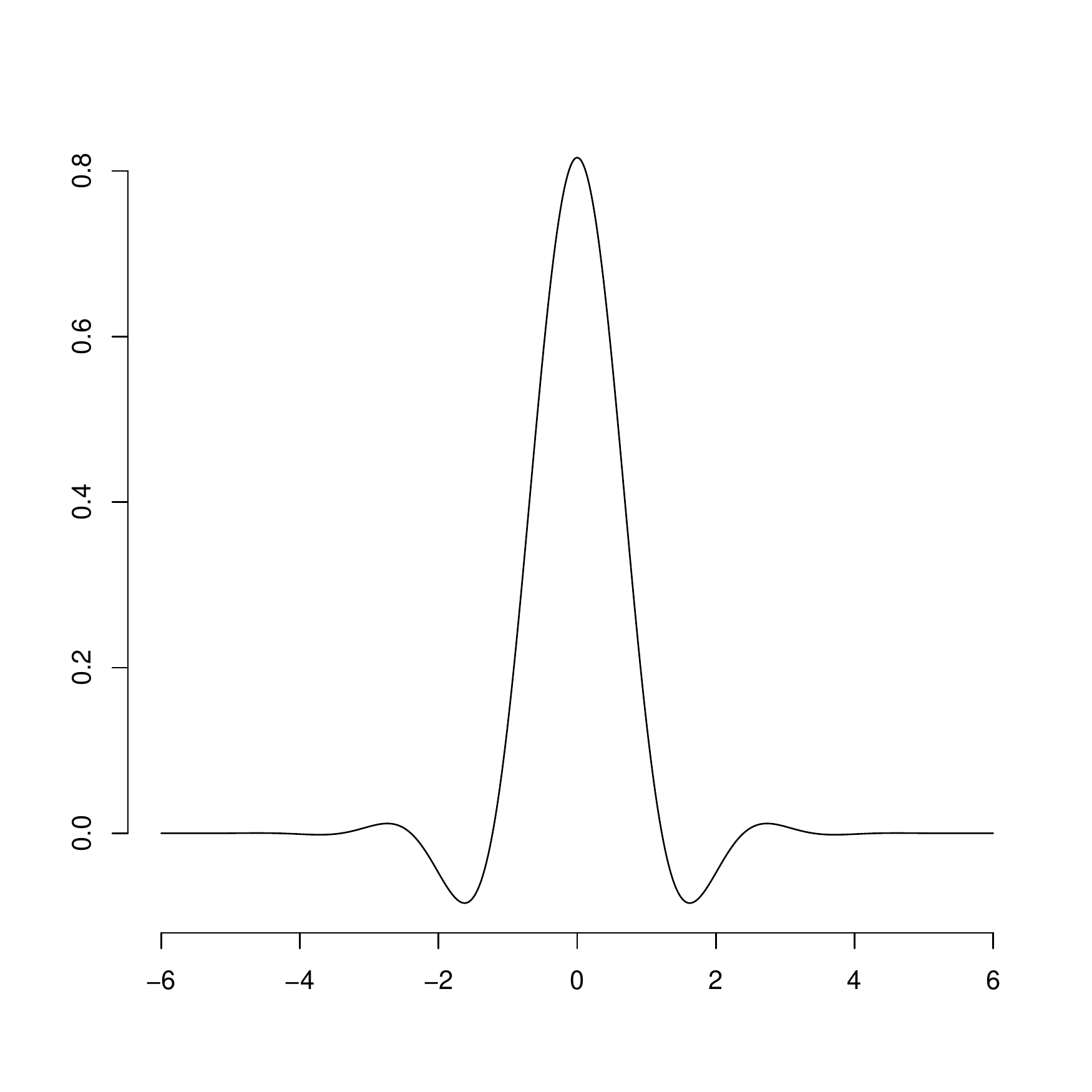} 
\caption{The solution to the fourth-order heat equation \eqref{eqsubFren}.}
\label{lab4ord}
\end{figure}

We examine here the composition of the Fresnel pseudo-process $F(t)$, $t>0$ with the first passage time $T_t$, $t>0$ (independent from $F$) of a Brownian motion $B$. The composition $F(T_t)$, $t>0$ has the relevant property that its one-dimensional distributions are true probability distributions as the next Theorem shows.

\begin{te}
The r.v. $F(T_t)$ has density
\begin{equation}
Pr\{ F(T_t) \in dx \}/dx = \frac{t}{\pi \sqrt{2}} \frac{t^2 + x^2}{t^4 + x^4}, \quad x \in \mathbb{R},\, t>0\label{trueD}
\end{equation}
and \eqref{trueD} solves the fourth-order equation 
\begin{equation}
\left( \frac{\partial^4}{\partial t^4} + \frac{\partial^4}{\partial x^4} \right) u=0. \label{truePDE}
\end{equation}
\end{te} 
\begin{proof}
We first of all prove that
\begin{equation}
 Pr\{ F(T_t) \in dx \} = dx \int_0^\infty \frac{1}{\sqrt{2\pi s}} \cos\left( \frac{x^2}{2s} - \frac{\pi}{4} \right) \frac{t e^{-\frac{t^2}{2s}}}{\sqrt{2\pi s^3}} ds
\end{equation}
satisfies equation \eqref{truePDE}. Since
\[ \left( \frac{\partial^2}{\partial t^2} - 2 \frac{\partial}{\partial s} \right) \frac{t e^{-\frac{t^2}{2s}}}{\sqrt{2\pi s^3}} =0 \]
we have that
\begin{align*}
\frac{\partial^4}{\partial t^4} Pr\{ F(T_t) \in dx \}/dx = & \int_0^\infty  \frac{1}{\sqrt{2\pi s}} \cos\left( \frac{x^2}{2s} - \frac{\pi}{4} \right) \frac{\partial^4}{\partial t^4} \frac{t e^{-\frac{t^2}{2s}}}{\sqrt{2\pi s^3}} ds\\
= & \int_0^\infty  \frac{2^2}{\sqrt{2\pi s}} \cos\left( \frac{x^2}{2s} - \frac{\pi}{4} \right) \frac{\partial^2}{\partial s^2} \frac{t e^{-\frac{t^2}{2s}}}{\sqrt{2\pi s^3}} ds\\
= & \int_0^\infty  \frac{\partial^2}{\partial s^2} \left[ \frac{2^2}{\sqrt{2\pi s}} \cos\left( \frac{x^2}{2s} - \frac{\pi}{4} \right) \right] \frac{t e^{-\frac{t^2}{2s}}}{\sqrt{2\pi s^3}} ds\\
= & - \frac{\partial^4}{\partial x^4} \int_0^\infty  \frac{1}{\sqrt{2\pi s}} \cos\left( \frac{x^2}{2s} - \frac{\pi}{4} \right) \frac{t e^{-\frac{t^2}{2s}}}{\sqrt{2\pi s^3}} ds.
\end{align*}
We now prove result \eqref{trueD}.
\begin{align*}
Pr\{ F(T_t) \in dx \}/dx = & \frac{t}{2^2 \pi} \int_0^\infty \left( e^{i\frac{x^2}{2s} - \frac{t^2}{2s} - i\frac{\pi}{4}} +  e^{-i\frac{x^2}{2s} - \frac{t^2}{2s} + i\frac{\pi}{4}} \right) \frac{ds}{s^2}\\
= & \frac{t}{2\pi} \left( \frac{e^{-i \frac{\pi}{4}}}{t^2 - ix^2} + \frac{e^{i \frac{\pi}{4}}}{t^2 + ix^2}\right)\\
= & \frac{t\, e^{-i \frac{\pi}{4}}}{2\pi} \left( \frac{1}{t^2 - ix^2} + \frac{i}{t^2 + ix^2}\right)\\
= & \frac{t\, e^{-i \frac{\pi}{4}}}{2\pi} \left( \frac{(t^2 + x^2)(1+ e^{i \frac{\pi}{2}})}{t^4 + x^4}\right)\\
= & \frac{t}{2\pi} \left( \frac{(t^2 + x^2)(e^{-i \frac{\pi}{4}}+ e^{i \frac{\pi}{4}})}{t^4 + x^4}\right) =  \frac{t}{\sqrt{2} \pi} \frac{t^2 + x^2}{t^4 + x^4}.
\end{align*}
\end{proof}

\begin{figure}[ht]
\centering
\includegraphics[scale=.5]{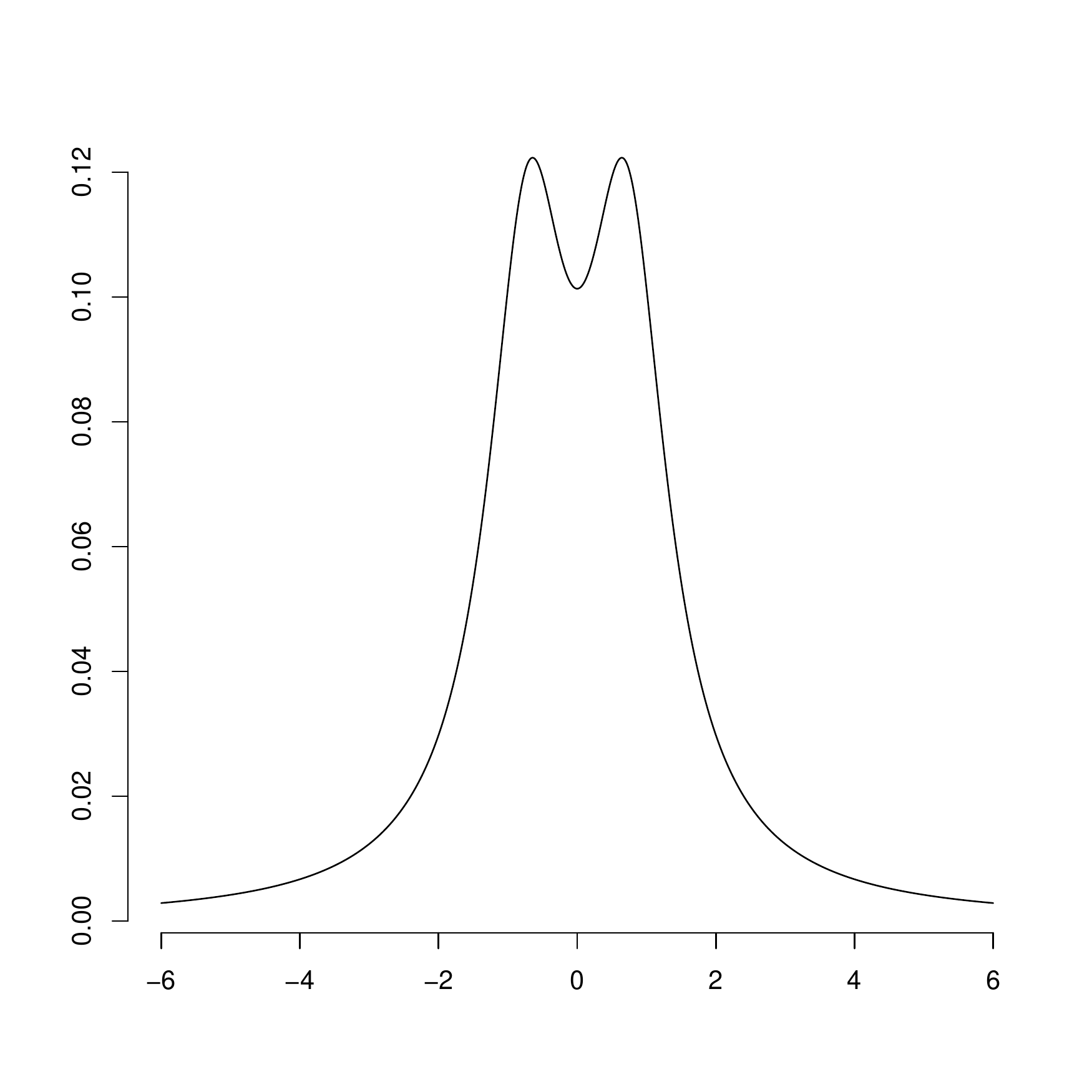} 
\caption{The density of the double-Cauchy r.v. $F(T_t)$.}
\end{figure}

\begin{os}
\normalfont
We can confirm result \eqref{trueD} by means of a different approach. Since equation \eqref{truePDE} can be written as 
\begin{equation}
\left( \frac{\partial^2}{\partial t^2} + e^{i\frac{\pi}{4}} \frac{\partial^2}{\partial x^2} \right) \left( \frac{\partial^2}{\partial t^2} + e^{-i\frac{\pi}{4}} \frac{\partial^2}{\partial x^2} \right) u = 0 \label{eqLap}
\end{equation}
each component of \eqref{eqLap} can be reduced to  a Laplace equation by means of the time transformation $t^\prime = e^{\pm \frac{\pi}{4}} t$. Therefore, the solution to \eqref{eqLap} can be organized as
\begin{equation}
Pr\{ F(T_t) \in dx \}/dx = \frac{1}{2\pi} \left[ \frac{t e^{i\frac{\pi}{4}}}{(t e^{i\frac{\pi}{4}})^2 + x^2} + \frac{t e^{-i\frac{\pi}{4}}}{(t e^{-i\frac{\pi}{4}})^2 + x^2} \right] = \frac{t}{\pi \sqrt{2}} \frac{t^2 + x^2}{t^4 + x^4}.
\end{equation}
This is tantamount to considering a Cauchy process $C$ whose time either flows on the positive or negative imaginary axis. The direction of time is initially chosen with equal probability.
\end{os}

\begin{te}
For the $n$-th order Fresnel iterated pseudo-process 
\[ \mathcal{F}_n(t) = F_1(|F_2(|\ldots F_{n+1}(t)  \ldots |)|), \quad t>0 \]
the governing equation of the measure density
\begin{equation}
\mu\{\mathcal{F}_n(t) \in dx \}/dx = 2^n \int_0^\infty \ldots \int_0^\infty \prod_{j=1}^{n} \frac{ds_{j+1}}{\sqrt{2\pi s_{j+1}}}\cos\left( \frac{s^2_{j}}{2s_{j+1}} - \frac{\pi}{4}\right)  d_{s_j} \label{densmulop}
\end{equation}
with $x=s_1$ and $t=s_{n+1}$, satsfies the equation
\begin{equation}
\frac{\partial^2 u}{\partial t^2}(x,t) = - 2^{-2(2^{n+1} -1)} \frac{\partial^{2^{n+2}}u}{\partial x^{2^{n+2}}}(x,t), \quad x \in \mathbb{R},\, t>0 \label{pdeOOOO}
\end{equation}
subject to the initial conditions
\begin{equation}
\left\lbrace 
\begin{array}{l} u(x,0)=\delta(x),\\u_t(x,0)=0.  \end{array}
\right .
\end{equation}
\end{te}
\begin{proof}
One can prove the statement  of the Theorem by evaluating the Fourier transform of \eqref{densmulop} by successively applying formula \eqref{jkl}. The Fourier integral w.r.t. $x$ yields
\begin{equation}
2^n \int_0^\infty \ldots \int_0^\infty \cos \frac{\beta^2 s_1}{2} \frac{ds_{1}}{\sqrt{2\pi s_{2}}}\cos\left( \frac{s^2_{1}}{2s_{2}} - \frac{\pi}{4}\right) \ldots \frac{ds_{n}}{\sqrt{2\pi t}}\cos\left( \frac{s^2_{n}}{2t} - \frac{\pi}{4}\right).
\end{equation}
The integral w.r.t. $s_1$ yields
\begin{align*}
2 \int_0^\infty \cos \frac{\beta^2 s_1}{2} \frac{1}{\sqrt{2\pi s_{2}}}\cos\left( \frac{s^2_{1}}{2s_{2}} - \frac{\pi}{4}\right) ds_{1} = & \int_{-\infty}^{+\infty} \cos \frac{\beta^2 s_1}{2} \frac{1}{\sqrt{2\pi s_{2}}}\cos\left( \frac{s^2_{1}}{2s_{2}} - \frac{\pi}{4}\right) ds_{1}\\
= & \int_{-\infty}^{+\infty} e^{-\frac{\beta^2 s_1}{2}} \frac{1}{\sqrt{2\pi s_{2}}}\cos\left( \frac{s^2_{1}}{2s_{2}} - \frac{\pi}{4}\right) ds_{1}\\
= & \cos \left( \left(\frac{\beta^2}{2}\right)^2 \frac{s_2}{2}\right).
\end{align*}
By iterating this procedure we arrive at the final relult
\[ \int_{-\infty}^{+\infty} e^{i \beta x} \mu\{ \mathcal{F}(t) \in dx \} = \cos \left(2t \left( \frac{\beta}{2}\right)^{2^{n+1}} \right). \]
With this at hand result  \eqref{pdeOOOO} immediately follows.
\end{proof}

\begin{figure}[ht]
\centering
\includegraphics[scale=0.85]{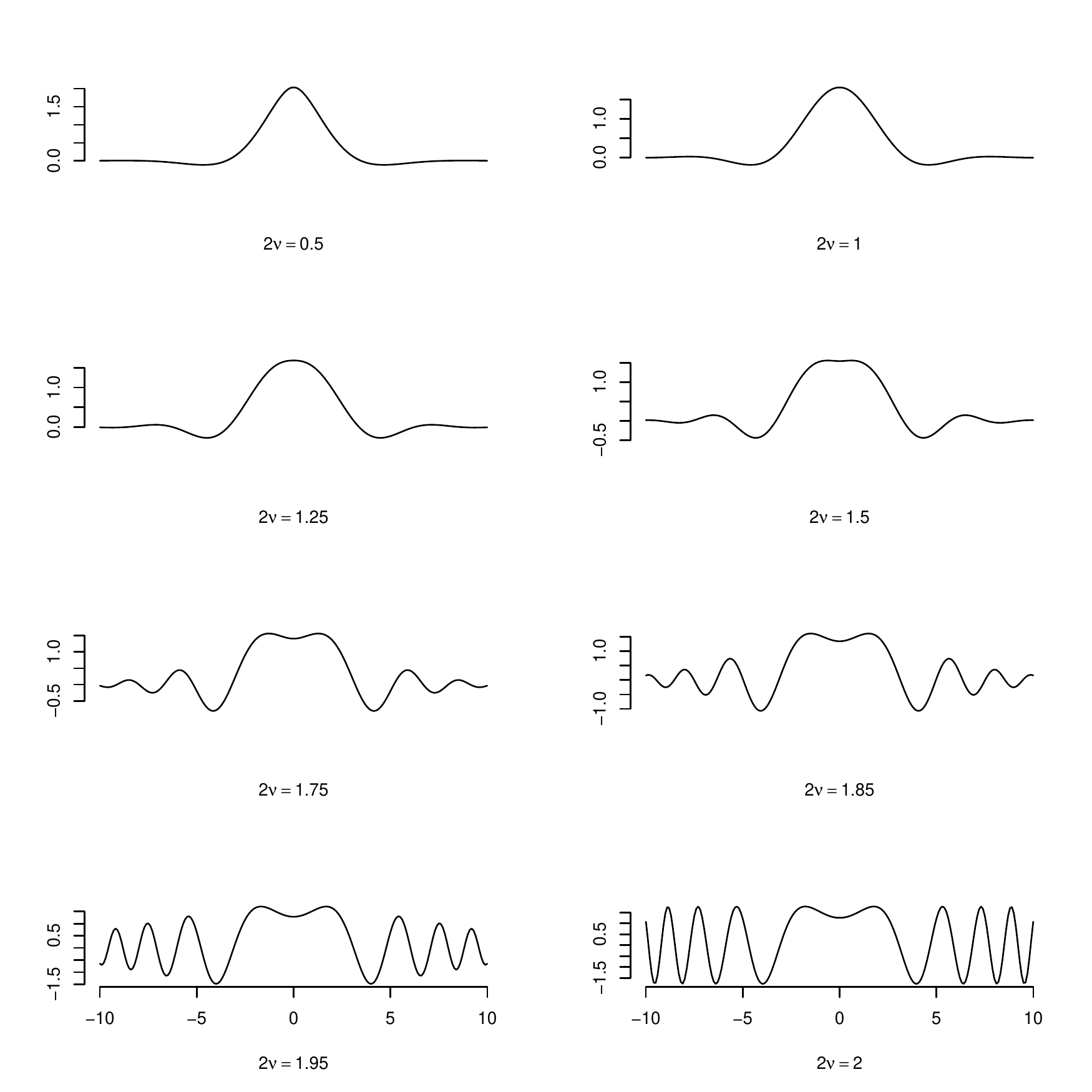}
\caption{The table of fundamental solutions of the Fresnel fractional equation \eqref{pdfgh} with different orders $2\nu$ of fractionality.}
\label{figall}
\end{figure}

\end{document}